\documentclass[11pt,reqno]{amsart}
\usepackage[colorlinks=true,linkcolor=blue,urlcolor=blue,citecolor=blue]{hyperref}

\usepackage{amssymb,amsmath,amsthm}
\usepackage{mathrsfs}
\usepackage{mathtools}
\usepackage{color}
\usepackage{marginnote}
\usepackage{graphicx}

\numberwithin{equation}{section}
\newtheorem{thm}{Theorem}[section]
\newtheorem{df}[thm]{Definition}

\newtheorem{lem}[thm]{Lemma}

\theoremstyle{definition}
\newtheorem{rem}[thm]{Remark}

\let\oldproofname=\proofname
\renewcommand{\proofname}{\rm\bf{\oldproofname}}

\newcommand{\N}{\mathbb{N}}

\newcommand{\R}{\mathbb{R}}

\newcommand{\cD}{\mathcal{D}}

\newcommand{\cF}{\mathcal{F}}
\newcommand{\cG}{\mathcal{G}}
\newcommand{\cH}{\mathcal{H}}

\newcommand{\cO}{\mathcal{O}}

\newcommand{\cR}{\mathcal{R}}

\newcommand{\cU}{\mathcal{U}}
\newcommand{\cV}{\mathcal{V}}
\newcommand{\cW}{\mathcal{W}}

\newcommand{\dd}{\,{\rm d}}
\newcommand{\D}{{\rm d}}

\newcommand{\DS}{\displaystyle}
\newcommand{\QED}{\mbox{}\hfill$\Box$}
\renewcommand{\:}{\thinspace :}

\title[Propagation fronts in a simplified model of tumor growth]
{Propagation fronts in a simplified model\\ of tumor growth with degenerate\\
cross-dependent self-diffusivity}

\author[Th. Gallay]{Thierry Gallay}
\address{Institut Fourier, Universit\'e Grenoble Alpes, CNRS, 100 rue des Maths -- 
38610 Gi\`eres, France}
\email{Thierry.Gallay@univ-grenoble-alpes.fr}

\author[C. Mascia]{Corrado Mascia}
\address{Dipartimento di Matematica Guido Castelnuovo, Sapienza, Universit\`a di Roma, 
P.le Aldo Moro 5 -- 00185 Roma, Italia}
\email{corrado.mascia@uniroma1.it}


\begin{document}
\baselineskip15pt

\maketitle

\thispagestyle{empty}

\begin{abstract}
Motivated by tumor growth in Cancer Biology, we provide a complete
analysis of existence and non-existence of invasive fronts for the
{\it reduced Gatenby--Gawlinski model}
\begin{equation*}
 	\partial_t U \,=\, U\bigl\{f(U)-dV\bigr\}\,,\qquad
  	\partial_t V \,=\, \partial_x \left\{f(U)\,\partial_x V\right\} + r V f(V)\,,
\end{equation*}
where $f(u)=1-u$ and the parameters $d, r$ are positive. Denoting by $(\cU,\cV)$ the 
traveling wave profile and by $(\cU_\pm,\cV_\pm)$ its asymptotic states at $\pm\infty$,
we investigate existence in the regimes
\begin{equation*}
	\begin{aligned}	
	&d>1:	&\;	& \bigl(\cU_-,\cV_-\bigr)=\bigl(0,1\bigr) 	
        &\textrm{and}\;\;\ \bigl(\cU_+,\cV_+\bigr)=\bigl(1,0\bigr),\\
	&d<1:	&\;	&\bigl(\cU_-,\cV_-\bigr)=\bigl(1-d,1\bigr) 	
        &\textrm{and}\;\;\ \bigl(\cU_+,\cV_+\bigr)=\bigl(1,0\bigr),
	\end{aligned}
\end{equation*}
which are called, respectively, {\it homogeneous invasion} and {\it
heterogeneous invasion}. In both cases, we prove that a propagating
front exists whenever the speed parameter $c$ is strictly positive.
We also derive an accurate approximation of the front profile in 
the singular limit $c \to 0$. 
\end{abstract}

\section{Introduction}\label{sec1}

Biological invasion is one of the basic features of Nature and its potentiality
to modify its structure and its inherent vitality. Sometimes invasion of a
new species can be regarded as a positive event, sometimes as a negative one
depending on the property of the intruder and the invaded, see
\cite{LewiPetrPott16}. Here, motivated by Cancer Biology, we focus on the
appearance of invasion in the form of propagating fronts in {\it tumor growth}.
Precisely, we present a rigorous mathematical analysis of a reaction-diffusion system
composed by two differential equations, for which we prove the existence of 
traveling wave solutions that can be interpreted as {\it invasion fronts} of a 
cancerous tissue into a healthy one.  We urge the reader to pay attention 
to the specific form of the nonlinear diffusion term in our system and 
the consequences it has on the set of admissible propagation speeds. 

\subsection{Genesis of the model}

The original motivation is the analysis of the so-called {\it acid-mediated tumor
  growth}, proposed by Otto Warburg as a mechanism responsible for tumor increase
\cite{WarbWindNege27}.  Precisely, the so-called {\it Warburg effect} refers to the
observation that --even in aerobic conditions-- cancer cells tend to favor metabolism via
glycolysis rather than the more efficient oxidative phosphorylation pathway, usually
preferred by most other cells of the body \cite{AlfaEtAl14}.  A simplified mathematical
description for such a mechanism has been proposed in \cite{GateGawl96}.  After an
appropriate rescaling and using the notation $f(s)=1-s$, the one-dimensional version of
the model reads as
\begin{equation}\label{GGsys}
	\left\{\begin{aligned}
 	 \partial_t U \,&=\, U\bigl\{f(U) - dW\bigr\}\,, \\ 
  	\partial_t V \,&=\, \partial_x \bigl\{f(U)\,\partial_x V\bigr\} + r V f(V) \,, \\
  	\partial_t W \,&=\, a\,\partial_x^2 W + b(V-W)\,,
	\end{aligned}\right.
\end{equation}
where $U = U(x,t)$ represents the (normalized) population of healthy cells, $V=V(x,t)$ is
the (normalized) population of tumor cells, and $W=W(x,t)$ is the concentration of lactic
acid. The reaction-diffusion system \eqref{GGsys} is a sound description of the 
acid-mediated tumor growth mechanism.  In what follows, we refer to \eqref{GGsys} as
the {\bf complete Gatenby--Gawlinski model} to distinguish it from a corresponding reduced
version to be introduced later.

The basics of such a modeling is rather clear.  Firstly, healthy cells --denoted by $U$--
have a certain reproduction level (supposed logistic with rate $1$, for simplicity) and
are deteriorated by the acid, following the standard mass action law with kinetic constant
$d$. Secondly, the tumor cells $V$ have the capability of spreading at a rate that
depends on the quantity of healthy cells $U$, and they also reproduce according to a
logistic law, with a different rate denoted by $r$. A rough justification of the
dependence on $U$ in the coefficient of $\partial_x V$ is that tumor cells --possessing a
high-degree of invasiveness-- can hardly move when the density of healthy cells is high.
Specifically, the coefficient is null (no motion of cancerous tissue) if the healthy cells
are at carrying capacity. Finally the concentration $W$ of lactic acid undergoes 
diffusion at constant rate $a$, and is increased proportionally to the unknown $V$, 
with kinetic constant $b$, until it reaches the saturation level $W = V$. Let us stress 
that the third unknown $W$ has no direct effect on the dynamics of the variable
$V$.  The system is meaningful when the parameters $a,b,d,r$ are all positive.
Modifications of the original model have also been considered by many authors:
among others, we quote here \cite{McGiEtAl14} (generalized Gatenby--Gawlinski
model), \cite{ArauFassSalv18} (linear diffusion in the tumor variable),
\cite{HoldRodr15} (effect of chemotherapy), \cite{MartEtAl10} (stromal
interaction), \cite{StinSuruMera15} (distinction between intracellular and
extracellular proton dynamics). An interesting feature of the Gatenby--Gawlinski
model \eqref{GGsys} is the numerical evidence of existence of invasive
propagation fronts, i.e. special solutions describing the invasion of the cancer 
cells into the healthy tissue. To our knowledge, no rigorous proof of the existence 
of such front is available so far, except in some limiting parameter regimes. 

To decrease the complexity of \eqref{GGsys}, we consider the reduced system
\begin{equation}\label{redGGsys}
	\left\{\begin{aligned}
	\partial_t U \,&=\, U \bigl\{f(U)- dV\bigr\}\,, \\ 
  	\partial_t V \,&=\, \partial_x \bigl\{f(U)\,\partial_x V\bigr\}+ r V f(V)\,,
	\end{aligned}\right.
\end{equation}
obtainable as a formal limit in the regime $a\sim\textrm{const.}$ and $b\to\infty$,
that is replacing the last dynamical equation with the trivial constitutive identity $W=V$ 
(see \cite{FasaHerrRodr09, MascMoscSime21, MoscSime19}). Such a model intends to 
describe the case in which the tumor cells act directly on the healthy tissue with 
no additional specific intermediate (at the level of ODE, an analogous system has been 
discussed in \cite{FassYang17}).  We refer to \eqref{redGGsys} as the {\bf reduced 
Gatenby--Gawlinski model}. It can be regarded as a simplified version of the system 
proposed in \cite{McGiEtAl14}, obtained by rescaling the space and choosing 
\begin{equation}\label{redGGgenGG}
	a_2=d, \quad r_2=r, \quad a_1=d_1=d_2=0, \quad D=1, \quad c=b,
\end{equation}
where the parameters $a_1,a_2, d_1, d_2, r_2$ and $D$ are as in \cite{McGiEtAl14}.
Incidentally, let us observe that there are other possible reductions of the same 
original complete model, which are obtained by considering different parameter 
regimes and could also be worth investigating.

The diffusion coefficient $f(U) = 1-U$ in the $V$-equation of system~\eqref{redGGsys} 
can be seen as an inherent defense process exerted by the healthy tissue in
response to the presence of the tumor. If the healthy cells are at their
carrying capacity, normalized to $1$, no invasion can occur; in contrast, tumor
starts growing whenever the density of healthy cells is lower than $1$, and in
absence of healthy tissue the tumor is free to permeate all the space.
Incidentally, we observe that $f(U)$ takes negative values when $U > 1$, so 
that the diffusion equation for $V$ becomes ill-posed in such a regime. 
In the subsequent traveling wave analysis, we restrict our attention to values 
$U \in (0,1)$, so that $f(U) > 0$. 

\subsection{Scalar reaction-diffusion equations with degenerate diffusion}

Propagation fronts attracted the interest of many researchers because they provide 
the simplest mathematical framework describing the process of biological invasion. 
Rigorous results concerning existence and asymptotic stability of traveling wave 
solutions for scalar equations of the form
\begin{equation}\label{degFKPP}
  \partial_t U \,=\, \partial_x \bigl\{\phi(U)\,\partial_x U\bigr\}+f(U)\,,
\end{equation}
with reaction term $f$ and nonlinear diffusion $\phi$, have been obtained under
various assumptions. When $\phi$ is a strictly positive constant, the diffusion
is linear and equation~\eqref{degFKPP} is semilinear. For $\phi$ dependent on $U$ and
attaining strictly positive values, the diffusion is nonlinear and non-degenerate.  Here,
we are interested in the case in which the function $\phi$ is non-negative and null at
some specific points, usually $U=0$, the archetypal example being the {\it porous medium
  equation} for which $\phi(s)\propto s^p$ for some $p>0$.  Let us stress that considering a
function $\psi$ such that $\psi'=\phi$, equation \eqref{degFKPP} can be rewritten as
\begin{equation*}
	\partial_t U \,=\, \partial_x^2 \psi(U)+f(U)\,,
\end{equation*}
corresponding to a porous medium equation with reaction. Existence of traveling wave 
solutions has been widely explored in such a context, see \cite{Aron80, Engl85, GildKers05, 
MalaMarc03, Newm80, SancMain94b, SancMain97, SancMain95, Sher10} for a single degeneration 
point, and \cite{DrabTaka20, MalaMarc05, MalaMarc03, Mans10b} for multiple degenerations.
Different forms of the multiplier function $\phi$, e.g. depending on $\partial_x U$, 
have also been considered (see \cite{AtkiReutRidl81,BengDepa18}).

In the scalar case, the typical existence statement --valid for linear, nonlinear
non-degenerate or degenerate diffusion-- can be rephrased as follows.  Let the function
$f$ be of logistic type, i.e. it has two zeros (say $0$ and $1$) and is positive in
between. Then the scalar equation \eqref{degFKPP} supports traveling waves
$U(x,t)=\mathcal{U}(x-ct)$ satisfying the asymptotic conditions $\mathcal{U}(-\infty)=1$
and $\mathcal{U}(+\infty)=0$ if and only if $c\geq c_\ast$, for some strictly positive $c_\ast$.
Moreover, the solution to the initial value problem with Heaviside-like data, characterized
by a sharp jump from 0 to 1, converges in an appropriate sense to the traveling wave
connecting the states $0$ and $1$ and moving at critical speed $c_\ast$. 

For traveling waves of linear diffusion equations, stability analysis is a
classical subject dating back to the pioneering papers by Fisher \cite{Fish37}
and Kolmogorov, Petrovskii and Piscounov \cite{KolmPetrPisc37}. Much less is
known, however, in the case of degenerate diffusion equations. One of the few
results is contained in \cite{LeyvPlaz20}, following the general method outlined
in \cite{MeyrRadeSier14} for general reaction-diffusion systems.

\subsection{Reaction-diffusion systems with cross-dependent self-diffusivities}

As expected, for systems, the situation is less clear. 
First of all, it is necessary to agree on the terminology.
Let us consider, for simplicity, a $2\times2$ reaction-diffusion system of the form
\begin{equation}\label{2x2}
  \left\{\begin{aligned}
  \partial_t u \,&=\, \partial_x\bigl\{\phi_{11}(u,v)\partial_x u+\phi_{12}(u,v)\partial_x 
  v\bigr\}+F(u,v)\,, \\ 
  \partial_t v \,&= \partial_x\bigl\{\phi_{21}(u,v)\partial_x u+\phi_{22}(u,v)\partial_x 
  v\bigr\}+ G(u,v)\,,
  \end{aligned}\right.
\end{equation}
for some diffusivities $\phi_{ij}$ with $i,j\in\{1,2\}$ and reaction terms $F, G$.

Many examples for \eqref{2x2} with a significant applied perspective can be provided.
Among others, the celebrated {\it Keller--Segel chemotaxis model} --proposed as a
description for the motion of bacteria towards some optimal environment \cite{KellSege71}--
fits into the class choosing
\begin{equation*}
	\phi_{11}(u,v)=a,\quad \phi_{12}(u,v)=0,\quad
	\phi_{21}(u,v)=-b\,\chi(u),\quad \phi_{22}(u,v)=\mu(u),
\end{equation*}
and $F(u,v)=-\kappa(u)v$, $G(u,v)= 0$ for some parameters $a, b>0$ and functions $\kappa$,
$\mu$, $\chi$.  The Keller--Segel system can be regarded as a prototype of ``exotaxis''
models because the gradient in the concentration of one species induces a flux of another
species.

A second example is \cite{ShigKawaTera79} --the progenitor of a long lineage--
where the terminology {\it cross-diffusion system} has been used to denote a
particular case of \eqref{2x2} characterized by the presence of the diffusion term
$\partial_x^2 \psi_i$ in place of $\partial_x(\phi_{i1}\partial_x u +\phi_{i2}\partial_x v)$ 
for $i=1,2$. In other words we assume that there exist $\psi_1, \psi_2$ such that
$\phi_{11} = \partial_u \psi_1$, $\phi_{12} = \partial_v\psi_1$, $\phi_{21} = \partial_u 
\psi_2$, $\phi_{22}=\partial_v\psi_2$, which is not the case in general. 

A simplified version of \eqref{2x2} is obtained by assuming the terms $\phi_{12}$ and
$\phi_{21}$ to be null, that is focusing on systems with {\it cross-dependent 
self-diffusivities}:
\begin{equation}\label{2x2self}
	\left\{\begin{aligned}
	\partial_t u \,&=\, \partial_x\left\{\phi_{11}(u,v)\partial_x u\right\}+F(u,v)\,, \\ 
  	\partial_t v \,&= \partial_x\left\{\phi_{22}(u,v)\partial_x v\right\}+ G(u,v)\,.
	\end{aligned}\right.
\end{equation}
Both species are submitted to self-diffusion with a diffusivity coefficient that, in
general, may depend on the other variable. The reduced Gatenby--Gawlinski model 
\eqref{redGGsys} fits into \eqref{2x2self} with the choices
\begin{equation*}
  \phi_{11}(u,v)=0,\quad \phi_{22}(u,v)=f(u),\quad F(u,v)=u \bigl\{f(u) - dv\bigr\}, 
  \quad G(u,v)= r v f(v).
\end{equation*}
Coming back to the topic of invasion fronts in the case of reaction-diffusion systems, a
huge difference with respect to the scalar case arises already for linear self-diffusion
because the dimension of the phase-space for the traveling wave ODE is strictly larger
than two.  Nonlinear non-degenerate self-diffusions make the analysis harder, but, in
principle, still manageable with an adapted strategy.  In contrast, when the diffusion
operator degenerates at some values, the situation becomes more involved with a
pivotal role played by an appropriate {\it desingularization procedure}, which 
will be explained below. 

Focusing on the case of cross-dependent self-diffusivities, after the pioneering
contribution of Aronson \cite{Aron80} devoted to a predator-prey system, the
attention moved toward the model proposed by Kawasaki et al. \cite{KawaEtAl97},
which attempts to provide a detailed description of the patterns generate by
some colonies of bacteria, called {\it Bacillus subtilis} (see
\cite{BenjCoheLevi00} for a comprehensive review on cooperative
self-organization of micro-organisms). This model is composed by two coupled
evolution equations for the population density $b$ and the concentration of
nutrient $n$.  Degenerate cross-dependent self-diffusion appears in the equation
for the bacteria $b$ and is proportional to the product of the two unknowns,
i.e.  $D(n,b) \propto n\cdot b$.  Investigations on existence of propagating
fronts in bacteria growth models --either from a purely analytical point of view
or from a numerical perspective-- have been performed in
\cite{FengZhou07,Mans10a,Mans17,SatnMainGardArmi01}. In particular, in
\cite{Mans17, SatnMainGardArmi01}, the existence result is very similar to the
one valid for the scalar case, including the existence of a traveling wave for
the critical speed.

Existence of propagation fronts for the complete Gatenby--Gawlinski model (and its 
modifications) is cogently supported by partial results and numerical calculations, 
see \cite{DaviEtAl18,FasaHerrRodr09,HoldRodrHerr14, McGiEtAl14, MoscSime19}. However, 
rigorous mathematical results are very limited if not completely missing. Again, a
distinguished feature of the model is the presence of cross-dependent self-diffusion. 
In contrast with the bacteria models, the cross-diffusion term $f(U)$ in the 
$V$-equation is a monotone {\em decreasing} function of the variable $U$, and our 
goal is to explore in detail the consequences of such kind of coupling, in the 
particular example of the reduced Gatenby--Gawlinski model \eqref{redGGsys}. 

\subsection{Statement of the main results}

A {\it traveling wave} for \eqref{redGGsys} is a solution of the form
\begin{equation*}
  U(x,t) \,=\, \cU(x-ct)\,, \quad V(x,t) \,=\, \cV(x-ct)\,,
\end{equation*}
where the parameter $c \in \R$ is the propagation speed. If $\xi := x-ct$ denotes 
the space variable in a comoving frame, the traveling wave profile formally 
satisfies the ODE system
\begin{equation}\label{Ode1}
  c\,\frac{\D\cU}{\D\xi} + \cU\bigl\{f(\cU) - d\,\cV\bigr\} \,=\, 0\,,\qquad
  \frac{\D}{\D\xi}\left\{f(\cU)\frac{\D\cV}{\D\xi}\right\} + c\,\frac{\D\cV}{\D\xi} 
  + r \cV f(\cV) \,=\, 0\,.
\end{equation}
A {\it propagation front} is a special type of traveling wave, enjoying the 
asymptotic conditions
\begin{equation}\label{BC1}
  \lim_{\xi\to\pm\infty} \bigl(\cU,\cV\bigr)(\xi) \,=\, \bigl(\cU_\pm,\cV_\pm\bigr)\,.
\end{equation}
The asymptotic values $\bigl(\cU_\pm,\cV_\pm\bigr)$ are forced to be
constant equilibria to system \eqref{redGGsys}. While numerical evidence
of existence of traveling waves has been provided in \cite{MascMoscSime21, MoscSime19},
no rigorous result was obtained so far. In what follows, we will reconsider the above 
definition of propagation front in order to incorporate
the presence of possible degeneracies.

When $d \neq 1$, system \eqref{redGGsys} has exactly four constant equilibria\:
the trivial state $(\bar U,\bar V) = (0,0)$, the healthy state $(1,0)$, the
cancerous state $(0,1)$ and the heterogeneous state $(1-d,1)$.  Note that the
last equilibrium is positive, hence biologically significant, if and only if
$d < 1$. In the limiting case $d = 1$, the system has only three uniform
equilibria.

In this paper, we mainly concentrate on the case $d>1$, which seems most relevant in
cancerology (considering $d$ as a measure of aggressiveness of the tumor),
and we look for traveling wave solutions that describe the invasion of the
healthy state by the infected state. In other words, we choose as asymptotic values
\begin{equation}\label{BC1bis}
  \bigl(\cU_-,\cV_-\bigr) \,=\, \bigl(0,1\bigr) \quad\textrm{and}\quad
  \bigl(\cU_+,\cV_+\bigr) \,=\, \bigl(1,0\bigr)\,.
\end{equation}
This situation is referred to as {\bf homogeneous invasion}. We also study
more succinctly the regime $0<d<1$, in which homogeneous invasion is
not possible. In that case, we focus on {\bf heterogeneous invasion} which 
corresponds to the asymptotic states
\begin{equation}\label{BC2}
  \bigl(\cU_-,\cV_-\bigr) \,=\, \bigl(1-d,1\bigr)
  \quad\textrm{and}\quad
  \bigl(\cU_+,\cV_+\bigr) \,=\, \bigl(1,0\bigr)\,.
\end{equation}
The particular case $d = 1$ is non-generic, and will not be considered here. 

It is important to keep in mind that the second equation in \eqref{redGGsys} is a {\it
  degenerate parabolic equation}, see \cite{SancMain95}, in the sense that the coefficient
$f(U)$ in front of the leading order term $\partial_x^2 V$ vanishes when $U = 1$. For that
reason, it is not clear at all that system~\eqref{redGGsys} has global classical solutions,
and a similar caveat applies to the ODEs \eqref{Ode1} satisfied by the traveling waves.
Hence, we adopt here the following definition, which is adapted from \cite{HilhEtAl08} and based
on the notion of weak solution.

\begin{df}\label{df:front}
The triple  $(\cU,\cV; c)$ is a {\bf propagation front} for system \eqref{redGGsys} 
connecting the asymptotic states $\bigl(\cU_-,\cV_-\bigr)$ and 
$\bigl(\cU_+,\cV_+\bigr)$ if
\begin{itemize}
  \item[\bf i)] $(\cU,\cV)\in C(\R\,;\,[0,1])\times C(\R\,;\,[0,1])$ 
  and $f(\cU)\dfrac{\D\cV}{\D\xi}\in L^2(\R)$;
  \item[\bf ii)] $(\cU,\cV)$ is a weak solution to \eqref{Ode1}, i.e. for all 
  $(\phi,\psi) \in C^1(\R)\times C^1(\R)$ with compact support
\begin{align}\label{frontdef1}
  &\int_{\R} \cU\left\{c\,\frac{\D\phi}{\D\xi} - \bigl[f(\cU)-d\,\cV\bigr]
  \phi\right\}\D \xi\,=\, 0\,,\\ \label{frontdef2}
  &\int_{\R} \left\{\left[f(\cU)\frac{\D\cV}{\D\xi} + c\,\cV\,\right]
  \frac{\D\psi}{\D\xi}  - r \cV f(\cV) \psi \right\}\D \xi\,=\, 0\,;
\end{align}
\item[\bf iii)] the asymptotic conditions \eqref{BC1} are satisfied.
\end{itemize}
The couple $(\cU,\cV)$ is the {\bf profile of the front} and the value $c$ is 
the {\bf speed of propagation}.
\end{df}

The main result of this paper is the following.

\begin{thm}\label{main1}
Assume that $d > 1$ and $r > 0$. For any $c > 0$, the reduced Gatenby--Gawlinski
system \eqref{redGGsys} has a propagation front $\bigl(\cU,\cV;c\bigr)$ connecting 
$(0,1)$ with $(1,0)$. This solution is unique up to translations and both components 
$\cU, \cV$ are strictly monotone functions of $\xi = x-ct$.
\end{thm}

The proof also provides detailed information on the behavior of the front profile as
$\xi \to \pm\infty$.  In particular, we can choose a translate of the wave such that
\begin{equation}\label{AsymUV1}
  \cU(\xi) \,=\, \alpha e^{\mu \xi} + \cO\bigl(e^{(\mu+\eta)\xi}\bigr)\,,
  \qquad \cV(\xi) \,=\, 1 - e^{\lambda \xi} + \cO\bigl(e^{(\lambda+\eta) 
  \xi}\bigr)\,, \qquad \hbox{as}\quad \xi \to -\infty\,,
\end{equation}
for some $\alpha > 0$, where
\begin{equation}\label{lamudef}
  \lambda \,=\, \frac12\bigl(-c + \sqrt{c^2 + 4r}\bigr) \,>\, 0\,,
  \qquad \mu \,=\, \frac{d-1}{c} \,>\, 0\,, \qquad 
  \eta \,=\, \min(\lambda,\mu)\,.
\end{equation}
Moreover, there exists $\beta > 0$ such that
\begin{equation}\label{AsymUV2}
  \cU(\xi) \,=\, 1 - \beta e^{-\gamma\xi} + \cO\bigl(e^{-2\gamma\xi}\bigr)\,,
  \qquad \cV(\xi) \,=\, \beta\,\frac{r+1}{d}\,e^{-\gamma\xi} 
  + \cO\bigl(e^{-2\gamma \xi}\bigr)\,,
\end{equation}
as $\xi \to +\infty$, where $\gamma = r/c$. 

\begin{figure}[ht]
  \begin{center}
 \begin{picture}(480,160)
  \put(20,0){\includegraphics[width=0.43\textwidth]{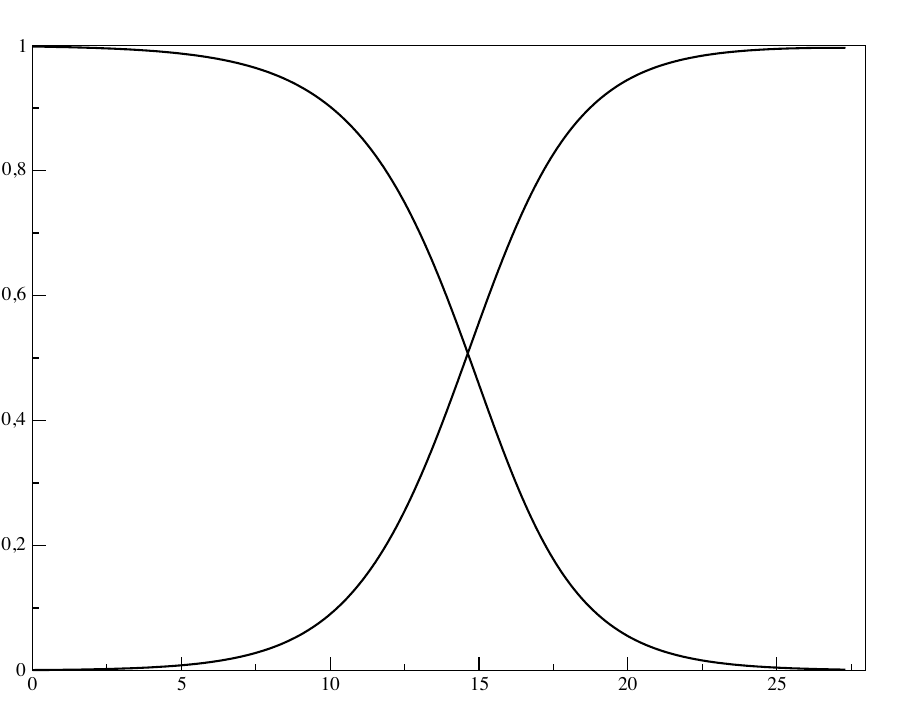}}
  \put(250,0){\includegraphics[width=0.43\textwidth]{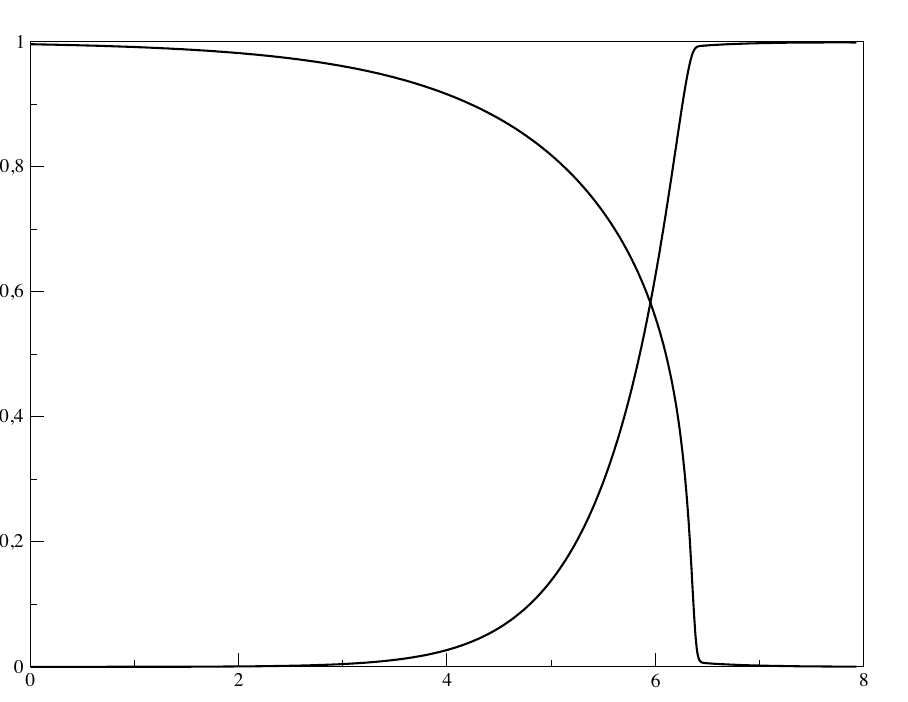}}
  \put(143,110){$\cU$}
  \put(90,110){$\cV$}
  \put(398,110){$\cU$}
  \put(357,110){$\cV$}
  \put(40,20){$c = 2.0$}
  \put(270,20){$c = 0.5$}
  \put(130,73){$\xrightarrow[\hspace{1.5cm}]{}$}
  \put(395,83){$\xrightarrow[\hspace{0.5cm}]{}$}
  \end{picture}
  \caption{{\small The profile $(\cU,\cV)$ of the propagation front describing homogeneous 
  invasion is represented as a function of $\xi = x-ct$ for $d = 2$, $r = 1$, 
  and for two different values of the speed parameter. Although the profile is smooth 
  in all cases, the right picture shows the appearance of a rather sharp edge when
  $c$ is small, a phenomenon that will be studied in Section~\ref{sec3}}. \label{fig1}}
  \end{center}
\end{figure}
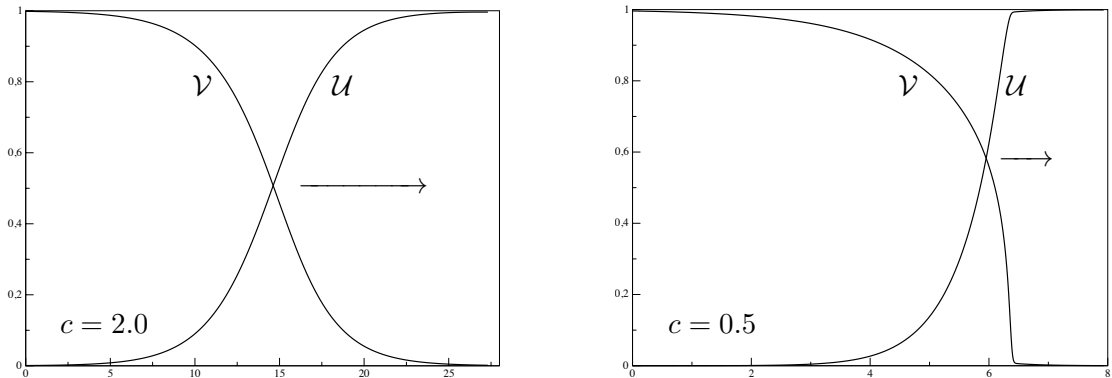

Most remarkably, Theorem~\ref{main1} shows that there exists {\bf no minimal speed} for
the propagation fronts of system \eqref{redGGsys}. This is in sharp contrast with what
happens for scalar equations involving a degenerate diffusion, see
\cite{SancMain94b,SancMain97} and also \cite{SatnMainGardArmi01} (for the Kawasaki
system). To elaborate on that, we consider a further formal reduction of system
\eqref{redGGsys}, which seems reasonable at least for solutions that evolve slowly in
time. In view of the first equation in \eqref{redGGsys}, one can expect that the first
component $U$ should stay close to $1-dV$ if $dV \le 1$ and to zero if $dV > 1$. Assuming
this to be exactly true, we obtain a scalar evolution equation for the second component
$V$\:
\begin{equation}\label{GGscalar}
  \partial_t V \,=\, \partial_x \left\{\phi(V)\partial_x V\right\}+ r V f(V)\,,
  \qquad \hbox{where}\quad \phi(V) \,:=\, \min(dV,1)\,.
\end{equation}
Strictly speaking, the results of \cite{SancMain97, MalaMarc05} do not apply to
\eqref{GGscalar} because the diffusion coefficient $\phi$ is only a Lipschitz
function of $V$. Disregarding that technical issue, we expect nevertheless that
equation \eqref{GGscalar} has monotone front solutions satisfying
$\cV(-\infty) = 1$, $\cV(+\infty) = 0$ if and only if $c \ge c_*$, for some
minimal speed $c_* = c_*(d,r) > 0$. Moreover, when $c = c_*$, the front profile
is ``sharp'' in the sense that there exists $\bar\xi \in \R$ such that
$\cV(\xi) = 0$ for all $\xi \ge \bar\xi$.  Quite surprisingly,
Theorem~\ref{main1} shows that the PDE system \eqref{redGGsys} behaves
differently\: propagation fronts exist for all positive speeds $c > 0$, no
matter how small, and all front profiles are smooth and strictly monotone. Sharp
fronts, which are typical for scalar equations with degenerate diffusion, do not
exist in system \eqref{redGGsys}.

Of course, this discrepancy means that the formal reduction leading to \eqref{GGscalar}
is not justified. In fact, if $(\cU,\cV)$ is the front profile given by 
Theorem~\ref{main1}, for some values of the parameters $d,r,c$, we can introduce
the {\em effective diffusion coefficient} $\phi : (0,1) \to (0,1)$ defined by
\begin{equation}\label{Deff}
  1 - \cU(\xi) \,=\, \phi\bigl(\cV(\xi)\bigr)\,, \qquad \xi \in \R\,.
\end{equation}
By construction, the second component $\cV$ of the front profile is a
traveling wave solution of the scalar equation \eqref{GGscalar} with
$\phi$ given by \eqref{Deff}. In particular we must have
$c \ge c_*(\phi,r)$, where $c_*$ is the minimal speed for the scalar
equation. Numerical calculations shows that the shape of the effective
diffusion coefficient $\phi$ depends strongly on the values of the
parameters $d,r,c$, and it often very different from the naive guess
$\phi(V) = \min(dV,1)$.  This is especially true when $c$ is small, in 
which case $\phi(V)$ is found to be extremely flat near the origin $V = 0$, 
see the discussion at the end of Section~\ref{sec3}.

Going back to \eqref{redGGsys}, in the case where $0 < d < 1$, the system has an additional,
biologically significant, equilibrium $(\bar U,\bar V) = (1-d,1)$ in which healthy and cancerous
cells coexist. In that case, it is natural to consider traveling waves that connect the
coexistence state to the healthy state given by \eqref{BC2}. We have the following
analogue of Theorem~\ref{main1}.

\begin{thm}\label{main2}
Assume that $0 < d < 1$ and $r > 0$. For any $c > 0$, the reduced Gatenby--Gawlinski
system \eqref{redGGsys} has a propagation front $\bigl(\cU,\cV;c\bigr)$ connecting 
$(1-d,1)$ with $(1,0)$. This solution is unique up to translations and both
components $\cU$ and $\cV$ are strictly monotone. Moreover, there is no 
propagation front connecting $(0,1)$ with $(1,0)$ in that case. 
\end{thm}

\begin{figure}[ht]
  \begin{center}
  \begin{picture}(480,160)
  \put(20,0){\includegraphics[width=0.43\textwidth]{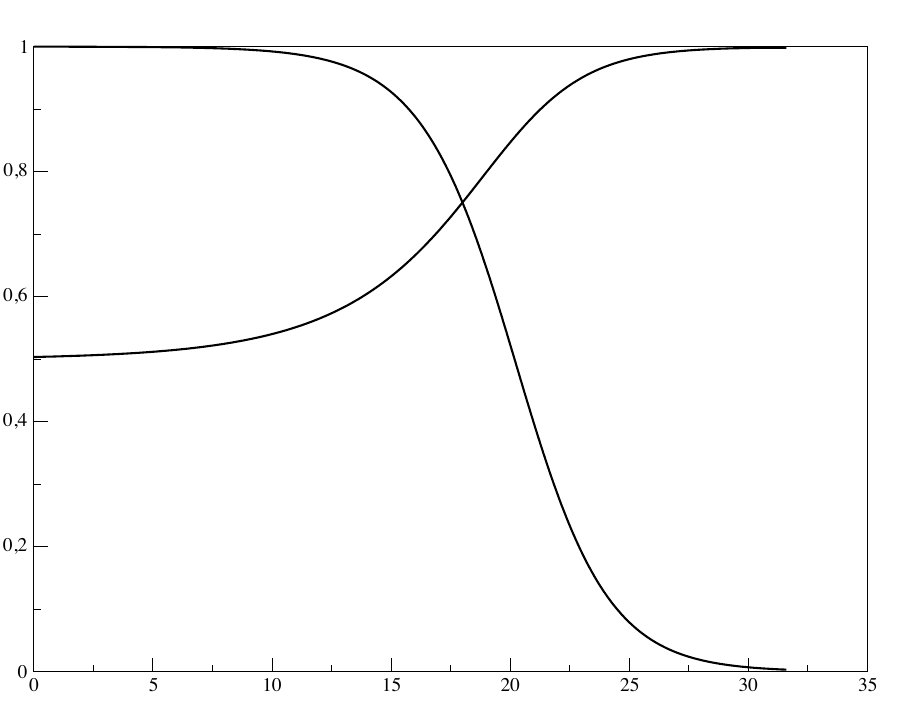}}
  \put(250,0){\includegraphics[width=0.43\textwidth]{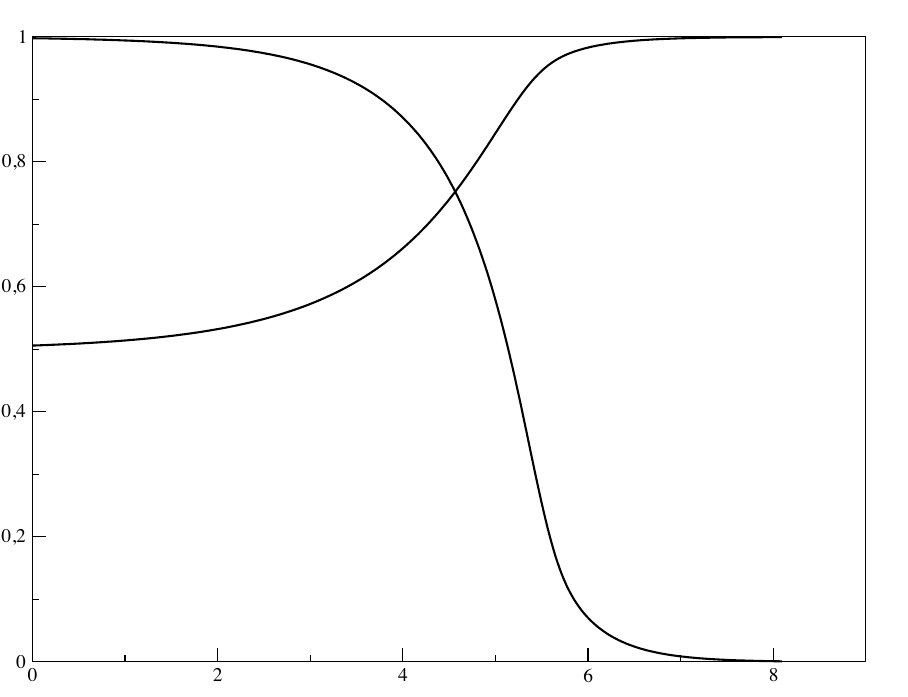}}
  \put(90,70){$\cU$}
  \put(144,40){$\cV$}
  \put(320,70){$\cU$}
  \put(370,40){$\cV$}
  \put(40,20){$c = 2.0$}
  \put(270,20){$c = 0.5$}
  \put(130,106){$\xrightarrow[\hspace{1.5cm}]{}$}
  \put(360,106){$\xrightarrow[\hspace{0.5cm}]{}$}
  \end{picture}
  \caption{{\small The profile $(\cU,\cV)$ of the front describing heterogeneous 
  invasion is represented for $d = 0.5$, $r = 1$, and the same values of $c$ 
  as in Fig.~\ref{fig1}. Here again the front profile becomes sharper 
  when $c$ is decreased (note that the horizontal scales are different
  in both pictures), but that phenomenon is less evident than for $d = 2$}. 
  \label{fig2}}
  \end{center}
\end{figure}
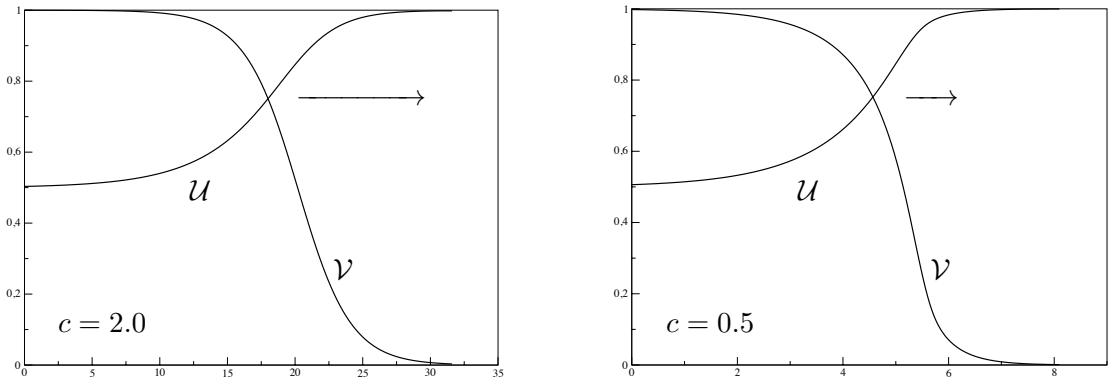

The rest of this article is organized as follows. Section \ref{sec2} is mainly
devoted to the proof of Theorem~\ref{main1}. Our starting point is the
desingularization of the ODE system \eqref{Ode1}, using a standard procedure
that was already known for scalar equations with degenerate diffusion. 
Propagation fronts are then constructed as heteroclinic connections between two
equilibria of the desingularized system. The unstable manifold of the infected
state is two-dimensional in that setting, which forces us to introduce an additional
shooting parameter, and an important part of our analysis relies on monotonicity
properties with respect to that shooting parameter. The proof of Theorem
\ref{main2} goes along the same lines, and is briefly presented in
Section~\ref{subsec27}. In Section~\ref{sec3}, we explore the limiting regime
where $c\to 0$, and we derive an asymptotic expansion of the front profile that
is remarkably accurate even at moderately small speeds.  Finally, we draw some
conclusions in Section~\ref{sec4} and we outline a few perspectives.

\subsection*{Acknowledgments.} The research of the first named author
is partially supported by the grant ISDEEC ANR-16-CE40-0013 of the French
Ministry of Higher Education, Research and Innovation. This work benefited 
from mutual invitations, on the occasion of which the hospitality of 
Universit\'e Grenoble Alpes and Sapienza, Universit\`a di Roma is gratefully 
acknowledged. 

\section{Existence of propagation fronts}\label{sec2}

This section is devoted to the proofs of Theorem~\ref{main1} and \ref{main2}.  We assume
that the reader is familiar with center manifold theory for ODEs and normal form theory.
All necessary material can be found in classical monographs such as \cite{CLW,GH,HI}, to
which we shall refer when needed.

In what follows, we fix the parameters $d > 0$ and $r > 0$ in system~\eqref{Ode1}.  All
quantities that appear in the proof depend on $d,r$, but for notational simplicity this
dependence will not be indicated explicitly.

\subsection{Preliminary results}
\label{subsec21a}

The underlying ODE system of \eqref{redGGsys} is
\begin{equation}\label{uniform}
  \frac{\D U}{\D t}=\, F(U,V;d)\,:=\,U\bigl(f(U)- d V\bigr)\,,\quad 
  \frac{\D V}{\D t}\,=\, G(U,V;r)\,:=\, r V f(V)\,.
\end{equation}
Regardless of the biological meaning of the variables, uniform
equilibria of \eqref{redGGsys} are
\begin{equation*}
  (\bar U,\bar V)=(0,0),\qquad (\bar U,\bar V)=(1,0),\qquad (\bar U,\bar V)=(0,1),
  \qquad (\bar U,\bar V)=(1-d,1)\,.
\end{equation*}
Computing the partial derivatives of $F$ and $G$, we determine the
linearized equation of \eqref{uniform} at $(\bar U,\bar V)$, which is
\begin{equation*}
	\frac{\D}{\D t}\begin{pmatrix} u \\ v\end{pmatrix} 
	=\mathbf{A}\begin{pmatrix} u \\ v\end{pmatrix} 
	\quad\textrm{where}\quad
	\mathbf{A}:=\begin{pmatrix} 1-2\bar U-d\bar V & -d\bar U \\ 0 & 
        r(1-2\bar V)\end{pmatrix}\,.
\end{equation*}
Hence, for the uniform equilibria, the following properties hold\:
\begin{itemize}
\item[\bf i)] $(0,0)$ is always an unstable node (eigenvalues $1$ and $r$);
\item[\bf ii)]  $(1,0)$ is always a saddle  (eigenvalues $-1$ and $r$);
\item[\bf iii)] $(0,1)$ is a saddle if $d<1$ and a stable node if $d>1$ 
(eigenvalues $1-d$ and $-r$);
\item[\bf iv)] $(1-d,1)$ is a stable node if $d<1$ and a saddle if $d>1$ 
(eigenvalues $d-1$ and $-r$).
\end{itemize}

Note the stability exchange between $(1-d,1)$ and $(0,1)$ when passing
the threshold $d=1$. In what follows, we look for propagation fronts with 
asymptotic states \eqref{BC1bis} if $d > 1$ and \eqref{BC2} if $d < 1$. 
The non-generic case $d = 1$ will not be considered. 

\begin{lem}\label{lem:positivec}
If $(\cU,\cV;c)$ is a propagation front in the sense of Definition~\ref{df:front}, 
then $\cV f(\cV) \in L^1(\R)$ and
\begin{equation}\label{cformula}
  c \,=\, r\int_{\R} \cV(\xi) f(\cV(\xi)) \dd\xi \,>\, 0\,.  
\end{equation}
\end{lem}

\begin{proof} 
Let $\chi : \R \to \R$ be a smooth non-decreasing function satisfying 
$\chi(x) = -1/2$ for $x \le -1$ and $\chi(x) = 1/2$ for $x \ge 1$. 
Given any $L > 1$, we consider the relation \eqref{frontdef2} in 
Definition~\ref{df:front}, with $\psi(\xi) = \psi_L(\xi) := \chi(\xi+L) - \chi(\xi-L)$. 
Note that $\psi$ is a smooth approximation of the characteristic function 
of the interval $[-L,L]$. Since $\cV$ is a continuous function having finite 
limits at infinity, we find
\[
  c \int_\R \cV\,\psi_L'\dd\xi ~\xrightarrow[L\to\infty]{}~
  c \bigl(\cV(-\infty) - \cV(+\infty)\bigr) \,=\, c\,.
\]
Moreover, using H\"older's inequality, we can bound
\[
  \biggl|\int_\R f(\cU)\cV' \psi_L'\dd\xi\biggr| \,\le\, \|\chi'\|_{L^2} 
  \left\{\biggl(\int_{-L-1}^{-L+1} \bigl|f(\cU)\cV'\bigr|^2\dd\xi\biggr)^{1/2}
  + \biggl(\int_{L-1}^{L+1} \bigl|f(\cU)\cV'\bigr|^2\dd\xi\biggr)^{1/2}
  \right\}\,,
\]
and the right-hand side converges to zero as $L \to +\infty$ because
$f(\cU)\cV' \in L^2(\R)$ by assumption. So we deduce from \eqref{frontdef2}
that
\[
  r\int_\R \cV f(\cV)\psi_L \dd \xi ~\xrightarrow[L\to\infty]{}~c\,,
\]
which gives the desired result since $\cV f(\cV) \ge 0$ and $\psi_L$ increases 
to $1$ as $L \to +\infty$. 
\end{proof}

\begin{lem}\label{lem:sharp}
If $(\cU,\cV;c)$ is a propagation front in the sense of Definition~\ref{df:front}, 
there exists a unique point $\bar\xi \in \R \cup \{+\infty\}$ such that\\[1mm]
1) $\cU,\cV \in C^\infty((-\infty,\bar\xi))$ and $0 < \cU(\xi),\cV(\xi) < 1$ 
for $\xi < \bar\xi$;\\[1mm]
2) If $\bar\xi < \infty$, then $\cU(\xi) = 1$ and $\cV(\xi) = 0$ for all 
$\xi \ge \bar\xi$. 
\end{lem}

In other words, the propagation front $(\cU,\cV;c)$ is {\bf smooth} if
$\bar\xi = +\infty$, and {\bf sharp} if $\bar\xi < +\infty$. In fact, we shall eventually
prove that the latter case cannot occur for system~\eqref{Ode1}, but at the moment we have
to consider both possibilities.

\begin{proof}
Since $\cU,\cV$ are continuous functions and $c > 0$ by Lemma~\ref{lem:positivec}, we
deduce from \eqref{frontdef1} that $\cU \in C^1(\R)$ and $\cU$ is a classical solution of
the first ODE in \eqref{Ode1}. In particular $\cU$ cannot vanish without being identically
zero, which would contradict the assumption that $\cU_+ = 1$, hence $\cU(\xi) > 0$ for all
$\xi \in \R$.  On the other hand, we know from \eqref{BC1} that $\cU(\xi) < 1$ when $\xi$
is large and negative. Thus either $\cU(\xi) < 1$ for all $\xi \in \R$, in which case we
set $\bar\xi = +\infty$, or there exists a (unique) point $\bar\xi \in \R$ such that
$\cU(\bar\xi) = 1$ and $\cU(\xi) < 1$ for all $\xi < \bar \xi$.

According to \eqref{frontdef2}, on the interval $I := (-\infty,\bar \xi)$ the function
$\cV$ is a weak solution of an elliptic ODE, so that $\cV$ is of class $C^2$ and satisfies
the second ODE in \eqref{Ode1} in the classical sense.  In fact, using \eqref{Ode1} and a
bootstrap argument, it is easy to verify that $\cU,\cV \in C^\infty(I)$.  Moreover, since
$0 \le \cV \le 1$ and since the nonlinear term $\cV f(\cV)$ vanishes when $\cV = 0$ and
$\cV = 1$, it is clear that $\cV$ cannot vanish on $I$ without being identically zero,
which would contradict the assumption that $\cV_- = 1$. Similarly, if $1 - \cV$ vanishes
somewhere on $I$, then $\cV \equiv 1$ on $I$; if $\bar\xi = +\infty$, this contradicts the
assumption that $\cV_+ = 0$, and if $\bar\xi < +\infty$ this implies that
$\cU'(\bar\xi) = d/c > 0$, which is of course impossible since $0 \le \cU \le 1$. The
proof of 1) is thus complete.

It remains to prove 2), assuming of course that $\bar\xi < +\infty$.  Since
$\cV(\xi) \ge 0$ by assumption, the first ODE in \eqref{Ode1} shows that
$c\,\cU'(\xi) = \cU(\xi)\bigl(\cU(\xi) + d\cV(\xi) - 1\bigr) \ge \cU(\xi)\bigl(\cU(\xi) 
- 1\bigr)$, with equality if and only if $\cV(\xi) = 0$. Since $\cU(\bar\xi) = 1$ and 
$\cU(\xi) \le 1$ for all $\xi \ge \bar\xi$, the only possibility is that $\cU(\xi) = 1$
and $\cV(\xi) = 0$ for all $\xi\ge\bar\xi$. This concludes the proof.
\end{proof}

\subsection{Desingularization of the ODE system}
\label{subsec21}

Here, we concentrate on the regime $d>1$ and we look for propagation fronts with
asymptotic states \eqref{BC1bis}. We know from Lemma~\ref{lem:positivec} that $c>0$, and
from Lemma~\ref{lem:sharp} that the profiles $\bigl(\cU,\cV\bigr)$ satisfy \eqref{Ode1} in
the classical sense on the interval $I = (-\infty,\bar\xi)$ for some
$\bar\xi \in \R \cup\{+\infty\}$. The ODE system \eqref{Ode1} degenerates in the limit
where $\xi \to \bar\xi$, which complicates the analysis.  Fortunately, as in scalar
equations \cite{Engl85}, it is possible to desingularize \eqref{Ode1} using a relatively
simple change of variables.

Given a solution $\bigl(\cU,\cV\bigr) : I \to (0,1)^2$ of \eqref{Ode1} 
satisfying \eqref{BC1}, we define a new independent variable $y = \Phi(\xi)$ 
by setting
\begin{equation}\label{ydef}
  \frac{\D y}{\D \xi} \,\equiv\, \Phi'(\xi) \,=\, \frac{1}{1 - \cU(\xi)}\,, 
  \qquad \hbox{for all~} \xi \in I = (-\infty,\bar\xi)\,.
\end{equation}
Since $\cU \in C^1(\R)$ and $U(\xi) \to 1$ as $\xi \to \bar\xi$, it is clear that 
\[
   \int_{-\infty}^0 \frac{1}{1 - \cU(\xi)}\dd\xi \,=\, +\infty
   \qquad\hbox{and}\quad
   \int_0^{\bar\xi} \frac{1}{1 - \cU(\xi)}\dd\xi \,=\, +\infty\,,
\]
no matter whether $\bar\xi < +\infty$ or $\bar\xi = +\infty$. This shows that
$\Phi :I \to \R$ is a smooth diffeomorphism, so that we can introduce the new 
dependent variables $(u,v)$ defined by
\begin{equation}\label{uvdef}
  u(y) \,=\, \cU\bigl(\Phi^{-1}(y)\bigr)\,, \quad
  v(y) \,=\, \cV\bigl(\Phi^{-1}(y)\bigr)\,, \qquad y \in \R\,.
\end{equation}
Using \eqref{ydef}, \eqref{uvdef}, it is straightforward to verify that the functions 
$u,v$ are solution of the desingularized system
\begin{equation}\label{Ode2}
	\left\{\begin{aligned}
 	c\,\frac{\D u}{\D y} + u(1-u)(1-u - d v) \,&=\, 0\,,\\
  	\frac{\D^2 v}{\D y^2}  + c\,\frac{\D v}{\D y} + r v(1-u)(1-v) \,&=\, 0\,,
	\end{aligned}\right.
\end{equation}
which is considered on the whole real line. The boundary conditions are unchanged: 
\begin{equation}\label{BC3}
  \lim_{y \to -\infty}\bigl(u(y),v(y)\bigr) \,=\, (0,1)\,,
  \qquad 
  \lim_{y \to +\infty}\bigl(u(y),v(y)\bigr) \,=\, (1,0)\,.
\end{equation}
It is important to observe that the desingularized system \eqref{Ode2} has many more
equilibria than the original system \eqref{Ode1}.  Indeed, in addition to the trivial
state $(u,v) = (0,0)$ and the infected stated $(u,v) = (0,1)$, system \eqref{Ode2} has a
continuous family of equilibria of the form $(u,v) = (1,v_\infty)$ for arbitrary
$v_\infty \in \R$.  Except for the healthy state $(u,v) = (1,0)$, those equilibria are an
artifact of the change of variables \eqref{ydef} and do not correspond to physically
meaningful situations.

If we introduce the additional variable $w = {\D v}/{\D y}$, we obtain
from \eqref{Ode2} the first-order system
\begin{equation}\label{Ode3}
	\left\{\begin{aligned}
  	c\frac{\D u}{\D y} \,&=\,  -u(1-u)(1-u - d v)\,, \\
  	\frac{\D v}{\D y}\,&=\, w\,, \\
  	\frac{\D w}{\D y} \,&=\, -c w - rv(1-u)(1-v)\,,
	\end{aligned}\right.
\end{equation}
which is the starting point of our analysis.  In the following sections, we 
consider solutions of \eqref{Ode3} that lie in the region $\cD \subset \R^3$ 
defined by
\begin{equation}\label{Ddef}
  \cD \,=\, \bigl\{(u,v,w) \in \R^3\,\big|\, 0 < u < 1\,,~
  0 < v < 1\,,~ w < 0\bigr\}\,.
\end{equation}
Indeed, the constraints $0 < u,v < 1$ were established in Lemma~\ref{lem:sharp}, 
and we shall see below that all front profiles also satisfy $w < 0$. 

\subsection{The unstable manifold of the infected state}
\label{subsec22}

The linearization of system \eqref{Ode3} at the infected state $S_- = (0,1,0)$ is
\begin{equation}\label{linS-}
  c\frac{\D u}{\D y} \,=\, (d-1)\,u\,, \quad \frac{\D z}{\D y}  \,=\, - w\,, 
  \quad \frac{\D w}{\D y} \,=\,  -rz-cw\,,
\end{equation}
where $z = 1 -v$. The equilibrium $S_-$ is thus hyperbolic, with two
positive eigenvalues $\lambda,\mu$ given by \eqref{lamudef}, and one
negative eigenvalue $\zeta = -\frac12\bigl(c + \sqrt{c^2 + 4r}\bigr)$.
In view of \eqref{BC3}, we are are interested in solutions that lie on
the two-dimensional unstable manifold of $S_-$.  It is a
straightforward task to compute an asymptotic expansion of all such
solutions in a neighborhood of $S_-$, see e.g. \cite[Chapter~3]{GH}.
Keeping only the solutions that belong to the region $\cD$ near $S_-$,
we obtain the following representation:

\begin{lem}\label{lem:unstab} Fix $c > 0$. For any $\alpha > 0$, 
the ODE system \eqref{Ode3} has a unique solution such that
\begin{align}\nonumber
  u(y) \,&=\, \alpha\,e^{\mu y} + \cO\Bigl(e^{(\mu + \eta)y}
  \Bigr)\,, \\ \label{asym1}
  v(y) \,&=\, 1 - e^{\lambda y} + \cO\Bigl(e^{(\lambda + \eta)y}\Bigr)\,, \\
  \nonumber
  w(y) \,&=\, - \lambda\,e^{\lambda y} + \cO\Bigl(e^{(\lambda + \eta)y}\Bigr)\,,
\end{align}
as $y \to -\infty$, where $\lambda,\mu$ are given by \eqref{lamudef}
and $\eta = \min(\lambda,\mu) > 0$. 
\end{lem}

\begin{rem}\label{rem:allsol}
Up to translations in the variable $y \in \R$, Lemma~\ref{lem:unstab}
describes all solutions of \eqref{Ode3} that converge to $S_-$ as
$y \to -\infty$ and belong to the region $\cD$ for sufficiently large
$y < 0$. To prove Theorem~\ref{main1}, our strategy is to study the 
behavior of those solutions as a function of the {\em shooting parameter}
$\alpha > 0$ and the {\em speed parameter} $c > 0$. When needed, 
we denote by $(u_{\alpha,c},v_{\alpha,c},w_{\alpha,c})$ the unique solution
of \eqref{Ode3} satisfying \eqref{asym1}. Among other properties, 
we shall use the fact that, for any $y_0 \in \R$, the solution
$(u_{\alpha,c}(y),v_{\alpha,c}(y),w_{\alpha,c}(y))$ depends continuously on the 
shooting parameter $\alpha$, uniformly in $y \in (-\infty,y_0]$. 
\end{rem}

The solution of \eqref{Ode3} satisfying \eqref{asym1} is not 
necessarily globally defined. The following result clarifies
under which condition the solution is global and stays in the 
region \eqref{Ddef} for all $y \in \R$. 

\begin{lem}\label{lem:invariance}
If the solution $(u_{\alpha,c},v_{\alpha,c},w_{\alpha,c})$ is defined on 
some interval $J = (-\infty,y_0)$ and satisfies 
$v_{\alpha,c}(y) > 0$ for all $y \in J$, then $(u_{\alpha,c}(y),
v_{\alpha,c}(y),w_{\alpha,c}(y)) \in \cD$ for all $y \in J$. 
\end{lem}

\begin{proof}
We denote $(u,v,w) = (u_{\alpha,c},v_{\alpha,c},w_{\alpha,c})$. 
Since the right-hand side of the first equation in \eqref{Ode3} vanishes 
when $u = 0$ and $u = 1$, it is clear that $0 < u(y) < 1$ for all 
$y \in J$. Next, assuming that $v(y)$ stays positive, we claim that
$w(y) < 0$ for all $y \in J$. Indeed, if this is not the case, we 
can find $y_1 < y_0$ such that $w(y_1) = 0$ and $w(y) < 0$ for all 
$y \in (-\infty,y_1)$. In particular, we have $0 < v(y) < 1$ for 
all $y \in (-\infty,y_1]$, and the last equation in \eqref{Ode3} 
shows that $w'(y_1) = -r(1-u(y_1))v(y_1)(1-v(y_1)) < 0$, which 
gives a contradiction. So $w(y) < 0$ for all $y \in J$, which 
implies that $0 < v(y) < 1$ for all $y \in J$. 
\end{proof}

For any $\alpha > 0$ and $c > 0$, we now define
\begin{equation}\label{Tdef}
  T(\alpha,c) \,=\, \sup\bigl\{y_0 \in \R\,\big|\,
  v_{\alpha,c}(y) > 0 \hbox{ for all } y < y_0\bigr\} \,\in\,
  (-\infty,+\infty]\,.
\end{equation}
According to Lemma~\ref{lem:invariance}, there are just two possibilities\:

\begin{itemize}

\item Either $T(\alpha,c) < +\infty$, in which case 
$v_{\alpha,c}(T(\alpha,c)) = 0$ and $w_{\alpha,c}(T(\alpha,c)) < 0$, 
so that $v_{\alpha,c}(y)$ becomes negative for some $y > T(\alpha,c)$. 
The corresponding value of the shooting parameter $\alpha$ must 
therefore be disregarded. 

\item Or $T(\alpha,c) = +\infty$, in which case the solution 
$(u_{\alpha,c},v_{\alpha,c},w_{\alpha,c})$ is global and stays in the 
region $\cD$ for all $y \in \R$. These are the solutions among 
which we want to find the traveling wave profiles satisfying
\eqref{BC3}. 

\end{itemize}

\subsection{Monotonicity with respect to the shooting parameter}
\label{subsec23}

A crucial observation is that the solutions of \eqref{Ode3} on 
the unstable manifold of $S_-$ are monotone functions of the 
shooting parameter $\alpha$. The precise statement is the 
following\:

\begin{lem}\label{lem:monotone} Fix $c > 0$. If $\alpha_2 > \alpha_1 > 0$, 
then $T(\alpha_2,c) \ge T(\alpha_1,c)$ and the solutions of 
\eqref{Ode3} defined by \eqref{asym1} satisfy
\begin{equation}\label{comp1} 
  u_{\alpha_2,c}(y) \,>\, u_{\alpha_1,c}(y)\,, \quad 
  v_{\alpha_2,c}(y) \,>\, v_{\alpha_1,c}(y)\,,  
\end{equation}
for all $y \in (-\infty,T(\alpha_1,c))$. 
\end{lem}

\begin{proof}
Fix $\alpha_2 > \alpha_1 > 0$. In a first step, we show that
inequalities \eqref{comp1} hold for all sufficiently large $y < 0$. 
From \eqref{asym1} we already know that $u_{\alpha_2,c}(y) - u_{\alpha_1,c}(y) 
\approx (\alpha_2-\alpha_1)\,e^{\mu y}$ as $y \to -\infty$, which proves 
the first inequality in \eqref{comp1} in the asymptotic regime. 
To establish the second inequality, we need a higher order expansion 
of the solutions on the unstable manifold of $S_-$, which can be obtained
as follows. For $i = 1,2$, we denote $u_i = u_{\alpha_i,c}$, $v_i = 
v_{\alpha_i,c}$, and we introduce the functions $\omega_i$ defined by
$\omega_i(y) = e^{-\lambda y}(1 - v_i(y))$, for $y < 0$ sufficiently 
large. A direct calculation shows that 
\begin{equation*}
  \omega_i'' + \delta \omega_i' + r \omega_i F\bigl(u_i,e^{\lambda y}\omega_i\bigr) 
   \,=\, 0\,,\qquad i = 1,2\,,
\end{equation*}
where $\delta = \sqrt{c^2 + 4r}$ and $F(u,\tilde v) = u + (1-u)\tilde v$. In addition, 
according to \eqref{asym1}, we have $\omega_i(y) \to 1$ and $\omega_i'(y) \to 0$ 
as $y \to -\infty$. We now consider the difference $\omega = \omega_1 - \omega_2$, 
which satisfies the inhomogeneous equation
\begin{equation}\label{omeq} 
  \omega'' + \delta \omega' + r\omega G \,=\, rf\,,
\end{equation}
where $G = G(y) := u_2 + (1-u_2)e^{\lambda y}(\omega_1+\omega_2)$ and  $f = f(y) := 
\omega_1(u_2-u_1)(1-e^{\lambda y}\omega_1)$. Integrating \eqref{omeq} and using the fact 
the $\omega(y) \to 0$ and $\omega'(y) \to 0$ as $y \to -\infty$, we obtain the
integral equation
\begin{equation}\label{omeqint} 
  \omega(y) \,=\, \frac{r}{\delta}\int_{-\infty}^y \Bigl(1 - e^{-\delta(y-z)}\Bigr)
  \Bigl(f(z) - \omega(z)G(z)\Bigr) \dd z\,, 
\end{equation}
which can be used to compute iteratively an asymptotic expansion of $\omega(y)$ as 
$y \to -\infty$. Since $G(y) = \cO(e^{\eta y})$ as $y \to -\infty$, where $\eta = 
\min(\lambda,\mu) > 0$, the leading order term is simply obtained by setting 
$\omega = 0$ in the right-hand side of \eqref{omeqint}.
If we observe that
\begin{equation*} 
  f(y) \,=\, \omega_1(y) \bigl(u_2(y) - u_1(y)\bigr) \Bigl(1 - e^{\lambda y}
  \omega_1(y)\Bigr) \,=\, (\alpha_2-\alpha_1)e^{\mu y} + \cO\Bigl(e^{(\mu + \eta)y}\Bigr)\,,
\end{equation*}
we thus find
\begin{equation}\label{omexp}
  \omega(y) \,\equiv\, \omega_1(y) - \omega_2(y) \,=\, 
  \frac{r(\alpha_2 - \alpha_1)}{\mu(\mu+\delta)}\,e^{\mu y}
   + \cO\Bigl(e^{(\mu + \eta)y}\Bigr)\,, \qquad \hbox{as }y \to -\infty\,.
\end{equation}
Recalling that $\alpha_2 > \alpha_1$, we conclude that $\omega_1(y) > \omega_2(y)$ when
$y < 0$ is sufficiently large, which means that the second inequality in \eqref{comp1}
holds in the asymptotic regime. It also follows from the representation formula
\eqref{omeqint} that $\omega'(y) > 0$ when $y < 0$ is sufficiently large, and this in turn
implies that $v_2'(y) > v_1'(y)$ in that region. Summarizing, we have shown that there
exists $y_1 \in \R$ such that inequalities \eqref{comp1} hold for all
$y \in (-\infty,y_1]$. Moreover $v_{\alpha_2,c}'(y_1) > v_{\alpha_1,c}'(y_1)$.

In a second step, we prove that $T(\alpha_2,c) \ge T(\alpha_1,c)$ and that inequalities
\eqref{comp1} hold for all $y \in (-\infty, T(\alpha_1,c))$. Indeed, if this is not the
case, there exists $y_2 < \min\{T(\alpha_1,c),T(\alpha_2,c)\}$ such that both inequalities
in \eqref{comp1} hold on the interval $(y_1,y_2)$, but at least one becomes an equality at
$y = y_2$. Our strategy is to show that this is impossible. Denoting as before
$u_i = u_{\alpha_i,c}$ and $v_i = v_{\alpha_i,c}$ for $i = 1,2$, we observe that
\begin{equation*}
  v_i'' + c v_i' + \phi_i v_i \,=\, 0\,, \qquad
  \textrm{where}\quad \phi_i \,=\,  r(1-u_i)(1-v_i)\,.
\end{equation*}
As $0 < u_1(y) < u_2(y) < 1$ and $0 < v_1(y) < v_2(y) < 1$ for all 
$y \in [y_1,y_2)$, it is clear that $\phi_1(y) > \phi_2(y) > 0$ 
on that interval. We consider the ratio $\rho(y) = v_2(y)/v_1(y)$, 
which satisfies
\begin{equation}\label{rhoeq} 
  \rho''(y) + \left(c + \frac{2v_1'(y)}{v_1(y)}\right)\rho'(y) 
  - \bigl(\phi_1(y)-\phi_2(y)\bigr)\rho(y) \,=\, 0\,, \qquad y \in (y_1,y_2)\,.
\end{equation}
We know that $\rho(y_1) > 1$ and $\rho'(y_1) > 0$, because the point $y_1$ was chosen so
that $v_2(y_1) > v_1(y_1) > 0$ and $v_1'(y_1) < v_2'(y_1) < 0$.  On the other hand, the
differential equation \eqref{rhoeq} implies that the function $\rho$ cannot have a
positive local maximum on the interval $(y_1,y_2)$.  So we must have
$\rho(y) \ge \rho(y_1) > 1$ for all $y \in [y_1,y_2)$, and taking the limit $y \to y_2$ we
conclude that $\rho(y_2) = v_2(y_2)/v_1(y_2) > 1$.

To establish the first inequality in \eqref{comp1}, we observe that
$c u_i'(y) = \psi_i(y)\bigl(1 - u_i(y)\bigr)u_i(y)$ for all
$y \in [y_1,y_2]$, where $\psi_i = u_i + dv_i - 1$.  We thus have the
integral representation
\begin{equation*}
	h(u_i(y)) \,=\, h(u_i(y_1)) \,\exp\left(\frac{1}{c}\int_{y_1}^y 
	\psi_i(z)\dd z \right)\,, \qquad y \in [y_1,y_2]\,,\quad i = 1,2\,,
\end{equation*}
where $h(u) = u/(1-u)$.
In particular, 
\begin{equation*}
  	\frac{h(u_2(y_2))}{h(u_1(y_2))} \,=\, \frac{h(u_2(y_1))}{h(u_1(y_1))}
 	 \,\exp\left(\frac{1}{c}\int_{y_1}^{y_2} \bigl(\psi_2(z)-\psi_1(z)\bigr)
  	\dd z \right) \,>\, \frac{h(u_2(y_1))}{h(u_1(y_1))} \,>\, 1\,,
\end{equation*}
because $\psi_2(y) > \psi_1(y)$ on $(y_1,y_2)$. Thus $u_2(y_2) > u_1(y_2)$, 
so that both inequalities in \eqref{comp1} hold at $y = y_2$, which 
gives the desired contradiction. 
\end{proof}

We recall that the relevant values of the shooting parameter $\alpha > 0$ are those for
which $T(\alpha,c) = +\infty$. Since $T(\alpha,c)$ is a non-decreasing function of $\alpha$
by Lemma~\ref{lem:monotone}, the following definition is natural\:
\begin{equation}\label{alpha0def}
  \alpha_0(c) \,=\, \inf\bigl\{\alpha > 0 \,|\, T(\alpha,c) = +\infty
  \bigr\} \,\in\, [0,+\infty]\,.
\end{equation}
Two situations can occur, depending on the value of the speed 
parameter $c > 0$\:

\begin{lem}\label{lem:alpha0}
If $c \ge 2\sqrt{r}$, then $\alpha_0(c) = 0$. If $0 < c < 2\sqrt{r}$, 
then $0 < \alpha_0(c) < +\infty$. 
\end{lem}

\begin{proof}
If $c \ge 2\sqrt{r}$ ({\em strongly damped case}), we claim that $T(\alpha,c) = +\infty$
for all $\alpha > 0$, so that $\alpha_0(c) = 0$. Indeed, using the continuity properties
mentioned in Remark~\ref{rem:allsol}, it is easy to verify that, in the limit where
$\alpha \to 0$, the solution $(u_{\alpha,c},v_{\alpha,c})$ of \eqref{Ode2} given by
Lemma~\ref{lem:unstab} converges uniformly on compact intervals to $(0,v)$, where
$v : \R \to \R$ is the unique solution of the Fisher--KPP equation
\begin{equation}\label{FKPP}
  v'' + c v' + rv(1-v) \,=\, 0\,,
\end{equation}
normalized so that $e^{-\lambda y}(1 - v(y)) \to 1$ as $y \to -\infty$.  As is well known,
the Fisher--KPP front $v$ is positive when $c \ge 2\sqrt{r}$. Since $v_{\alpha,c}$ is an
increasing function of $\alpha$ by Lemma~\ref{lem:monotone}, we deduce that
$v_{\alpha,c}(y) > 0$ for all $y \in \R$ and all $\alpha > 0$, which means that
$T(\alpha,c) = +\infty$ for all $\alpha > 0$.

We next consider the opposite situation where $0 < c < 2\sqrt{r}$ ({\em weakly damped
  case}). In that case, the solution $v$ of the Fisher--KPP equation is no longer positive,
hence there exists $\bar y \in \R$ so that $v(\bar y) < 0$. By continuity, we have
$v_{\alpha,c}(\bar y) < 0$ when $\alpha > 0$ is sufficiently small, so that
$T(\alpha,c) < +\infty$ for all sufficiently small $\alpha > 0$. To conclude the proof, it
remains to show that $T(\alpha,c) = +\infty$ when $\alpha > 0$ is sufficiently large. It
is convenient here to define $y_0 = (\ln \alpha)/\mu \gg 1$ and to consider the shifted
quantities
\[
\begin{aligned}
  \hat u(y) \,&:=\, u_{\alpha,c}(y-y_0) \,=\, e^{\mu y} + 
  \cO\Bigl(e^{(\mu + \eta)y}\Bigr)\,, \\
  \hat v(y) \,&:=\, v_{\alpha,c}(y-y_0) \,=\, 1 - \beta\,e^{\lambda y} + 
  \cO\Bigl(e^{(\lambda + \eta)y}\Bigr)\,, \\
  \hat w(y) \,&:=\, w_{\alpha,c}(y-y_0) \,=\, - \beta\lambda\,e^{\lambda y} + 
  \cO\Bigl(e^{(\lambda + \eta)y}\Bigr)\,,
\end{aligned}
\qquad \hbox{as}\quad y \to -\infty\,,
\]
where $\beta = \alpha^{-\lambda/\mu} \to 0$ as $\alpha \to +\infty$.  On any interval of
the form $(-\infty,y_0]$, these functions converge uniformly to $(\chi,1,0)$ as
$\beta \to 0$, where $\chi$ is the unique solution of the differential equation
\begin{equation*}
  c\chi' \,=\, u(1-u)\bigl(d-1 + u\bigr)\,,
\end{equation*}
normalized so that $\chi(y) = e^{\mu y} + \cO(e^{2\mu y})$ as $y \to -\infty$.  It is
clear that $\chi$ is increasing and converges to $1$ as $y \to +\infty$. Given any small
$\epsilon > 0$, we can therefore choose $y_1 > 0$ large enough and $\beta > 0$ small
enough so that $\hat u(y_1) \ge 1-\epsilon$, as well as $\hat v(y) \ge 1-\epsilon$ and
$\hat v'(y) \ge -\epsilon$ for all $y \le y_1$. In the rest of the proof, we choose
$\epsilon = \epsilon_0/K$ where
\begin{equation}\label{Kdef}
  \epsilon_0 \,=\, \frac{d-1}{2d} \,\in\, \Bigl(0\,,\frac12\Bigr)\,,
  \qquad \hbox{and}\qquad K \,=\, 1 + \frac{1}{c} + \frac{2r}{d}\,>\, 1\,.
\end{equation}

Under these assumptions, we claim that $\hat v(y) \ge 1 - \epsilon_0$ for all $y \ge y_1$,
which implies that $T(\alpha,c) = +\infty$.  Indeed, as long as
$\hat v(y) \ge 1 - \epsilon_0$, the function $\hat u$ satisfies
\[
  \hat u'(y) \,=\, \frac{1}{c}\,\hat u(y)\bigl(1-\hat u(y)\bigr)
  \bigl(d\hat v(y) - 1 + \hat u(y)\bigr) \,\ge\, 
  \frac{\mu}{2}\,\hat u(y)\bigl(1-\hat u(y)\bigr)\,,
\]
because $d\hat v - 1 \ge d-1-d\epsilon_0 = d\epsilon_0 = c\mu/2$.  Integrating that
inequality for $y \ge y_1$ and recalling that $\hat u(y_1) \ge 1-\epsilon$, we obtain
\begin{equation}\label{hatubd}
  \frac{\hat u(y)}{1-\hat u(y)} \,\ge\, \frac{\hat u(y_1)}{1-\hat u(y_1)}
  \,e^{\mu(y-y_1)/2} \,\ge\, \frac{1-\epsilon}{\epsilon}\,e^{\mu(y-y_1)/2}\,,
\end{equation}
which shows that $1 - \hat u(y) \le 2\epsilon \,e^{-\mu(y-y_1)/2}$ as long as
$\hat v(y) \ge 1 - \epsilon_0$. Under that hypothesis, the 
function $\hat v$ satisfies a differential inequality of the form
$\hat v''(y) + c\hat v'(y) + \delta\,e^{-\gamma (y-y_1)} \ge 0$, where
$\delta = 2r\epsilon \epsilon_0$ and $\gamma = \mu/2$. Integrating
that inequality for $y \ge y_1$ and assuming for simplicity that
$\gamma \neq c$, we obtain
\[
  \hat v'(y) \,\ge\, \hat v'(y_1)\,e^{-c(y-y_1)} - \frac{\delta}{c-\gamma}
  \,\Bigl(e^{-\gamma (y-y_1)} - e^{-c(y-y_1)}\Bigr)\,,
\]
hence, recalling that $\hat v(y_1) \ge 1-\epsilon$ and $\hat v'(y_1) \ge 
-\epsilon$\: 
\begin{align}\nonumber
  \hat v(y) \,&\ge\, \hat v(y_1) + \frac{1 - e^{-c(y-y_1)}}{c}\,\hat v'(y_1)
  -\frac{\delta}{c-\gamma}\int_{y_1}^y \Bigl(e^{-\gamma (z-y_1)} - e^{-c(z-y_1)}
  \Bigr)\dd z \\ \label{hatvbd}
  \,&>\, \hat v(y_1) + \frac{1}{c}\,\hat v'(y_1) - \frac{\delta}{c\gamma}
  \,\ge\, 1 - \epsilon - \frac{\epsilon}{c} - \frac{4r\epsilon\epsilon_0}{
  c\mu} \,=\, 1 - K\epsilon\,,
\end{align}
where $K$ is defined in \eqref{Kdef}. Summarizing, inequalities 
\eqref{hatubd} and \eqref{hatvbd} together imply that the lower bound 
$\hat v(y) \ge 1-\epsilon_0$ holds in fact for all $y \ge y_1$,
so that $T(\alpha,c) = + \infty$ if $\alpha > 0$ is large enough. 
\end{proof}

\begin{rem}\label{rem:alpha0}
If $0 < c < 2\sqrt{r}$, then $T(\alpha_0(c),c) = +\infty$. Indeed, if this was not the
case, the solution $v_{\alpha_0(c),c}$ of \eqref{Ode2} defined in Lemma~\ref{lem:unstab}
would cross the origin (with a negative slope) at point
$\bar y = T(\alpha_0(c),c) < +\infty$. By continuity, $v_{\alpha,c}$ would also change
sign near $\bar y$ if $\alpha > \alpha_0(c)$ and $\alpha$ is sufficiently close to
$\alpha_0(c)$. Thus $T(\alpha,c) < +\infty$ for some $\alpha > \alpha_0(c)$, which
contradicts the definition of $\alpha_0(c)$. Summarizing, it follows from
Lemma~\ref{lem:alpha0} that $T(\alpha,c) = +\infty$ for all $\alpha > 0$ when
$c \ge 2\sqrt{r}$, and $T(\alpha,c) = +\infty$ if and only if $\alpha \ge \alpha_0(c)$
when $0 < c < 2\sqrt{r}$.
\end{rem}

\subsection{Asymptotic behavior as $y \to +\infty$}
\label{subsec24}

Using the results obtained so far, we now show that the solutions of \eqref{Ode2} on the
unstable manifold of $S_-$, when they stay in the region defined by \eqref{Ddef}, are
eventually monotone and converge therefore to finite limits as $y \to +\infty$.

\begin{lem}\label{lem:limits}
If $T(\alpha,c) = +\infty$, the following limits exist\:
\begin{equation}\label{uvlimits}
  u_\infty(\alpha,c) \,=\, \lim_{y \to +\infty} u_{\alpha,c}(y) 
  \in \{0,1\}\,, \qquad
  v_\infty(\alpha,c) \,=\, \lim_{y \to +\infty} v_{\alpha,c}(y) 
  \in [0,1)\,.
\end{equation}
Moreover, if $u_\infty(\alpha,c) = 0$, then $v_\infty(\alpha,c) = 0$. 
\end{lem}

\begin{proof}
Assume that $\alpha > 0$ and $c > 0$ are such that $T(\alpha,c) =
+\infty$, which means that the solution $(u,v,w) = (u_{\alpha,c},
v_{\alpha,c},w_{\alpha,c})$ of \eqref{Ode3} is global and stays in the 
region $\cD$ for all $y \in \R$. In particular, we have $0 < v(y) < 1$ and
$v'(y) = w(y) < 0$ for all $y \in \R$, which proves the existence
of the second limit in \eqref{uvlimits}. As for the function
$u$, there are two possibilities\:

\begin{itemize}

\item Either $u(y) + dv(y) > 1$ for all $y \in \R$, in which case the first equation in
  \eqref{Ode3} shows that $u'(y) > 0$ for all $y \in \R$. As $0 < u(y) < 1$, we deduce
  that $u(y)$ converges to some limit $u_\infty \in (0,1]$. Actually, since
  $(u_\infty,v_\infty)$ must be an equilibrium of \eqref{Ode3}, we necessarily have
  $u_\infty = 1$.

\item Or there exists $\bar y \in \R$ such that $u(\bar y) + dv(\bar y) = 1$. In that
  case, by the first equation in \eqref{Ode3}, we have $u'(\bar y) = 0 < - dv'(\bar y)$,
  and this implies that $u(y) + dv(y) < 1$ for all $y > \bar y$. The same argument also
  shows that $u(y) + dv(y) > 1$ for all $y < \bar y$, in agreement with
  \eqref{asym1}. Thus we conclude that $u'(y) > 0$ for all $y < \bar y$, and $u'(y) < 0$
  for all $y > \bar y$. In particular $u(y)$ converges to some limit $u_\infty \in [0,1)$
  as $y \to +\infty$, and we must have $u_\infty = v_\infty = 0$ since
  $(u_\infty,v_\infty)$ is an equilibrium of \eqref{Ode3}.

\end{itemize}
The proof of \eqref{uvlimits} is thus complete. 
\end{proof}

It is clear from Lemma~\ref{lem:monotone} that both limits $u_\infty$, $v_\infty$ in
\eqref{uvlimits} are non-decreasing functions of the shooting parameter $\alpha > 0$. Also,
the proof of Lemma~\ref{lem:alpha0} shows that, if $\alpha > 0$ is sufficiently large
(depending on $c$), we necessarily have $u_\infty = 1$ and $v_\infty > 0$.  This leads to
the following definition\:
\begin{equation}\label{alpha1def}
  \alpha_1(c) \,=\, \inf\bigl\{\alpha > \alpha_0(c) \,|\,
  u_\infty(\alpha,c) = 1\bigr\}\,.
\end{equation}

\begin{lem}\label{lem:alpha1}
For any $c > 0$ we have $0 < \alpha_1(c) < \infty$. Moreover
$\alpha_1(c) = \alpha_0(c)$ if $0 < c < 2\sqrt{r}$.
\end{lem}

\begin{proof}
Fix $c > 0$. We already observed that $u_\infty(\alpha,c) = 1$ when 
$\alpha > 0$ is sufficiently large, so that $\alpha_1(c) < +\infty$. 
If $c \ge 2\sqrt{r}$, so that $\alpha_0(c) = 0$ by Lemma~\ref{lem:alpha0},
we recall that the solution $(u_{\alpha,c},v_{\alpha,c})$ of \eqref{Ode2}
converges uniformly on compact sets to $(0,v)$ as $\alpha \to 0$,
where $v$ is the Fisher--KPP front. Since $v(y) \to 0$ as $y \to
+\infty$, we can choose $y \in \R$ so that $v(y) < 1/d$. If
$\alpha > 0$ is sufficiently small, we thus have $u_{\alpha,c}(y)
+ dv_{\alpha,c}(y) < 1$, and the proof of Lemma~\ref{lem:limits}
then shows that $u_\infty(\alpha,c) = 0$. Thus $\alpha_1(c) > 0$.

It remains to show that $\alpha_1(c) = \alpha_0(c)$ when $0 < c < 2\sqrt{r}$.  Indeed, if
$\alpha_1(c) > \alpha_0(c)$, we can take $\alpha \in (\alpha_0(c),\alpha_1(c))$ so that
the corresponding solution $(u,v,w) = (u_{\alpha,c},v_{\alpha,c},w_{\alpha,c})$ of
\eqref{Ode3} stays in $\cD$ for all $y \in \R$. Moreover $u_\infty = v_\infty = 0$ since
$\alpha < \alpha_1(c)$. We choose $\epsilon > 0$ small enough so that
$c < 2\sqrt{r}(1-\epsilon)$, and $\bar y > 0$ large enough so that $u(y) < \epsilon$ and
$v(y) < \epsilon$ for all $y \ge \bar y$. All solutions of the constant coefficient ODE
\[
  v''(y) + cv'(y) + r(1-\epsilon)^2 v(y) \,=\, 0\,, \qquad
  y \in \R\,,
\]
have infinitely many zeros in the interval $(\bar y,+\infty)$, and Sturm's comparison
theorem asserts that the function $v_{\alpha,c}$, which satisfies the second equation in
\eqref{Ode2} where $(1-u)(1-v) > (1-\epsilon)^2$, has a fortiori infinitely many zeros in
that interval, see e.g. \cite[Chapter~8]{CL}. This of course contradicts the assumption
that $\alpha > \alpha_0(c)$.
\end{proof}

\begin{rem}\label{rem:alpha1}
It follows from Remark~\ref{rem:alpha0} that $T(\alpha_1(c),c) = +\infty$ 
for any $c > 0$. It is also easy to verify that $u_\infty(\alpha_1(c),c) = 1$. 
Indeed, if this is not the case, we have $u_\infty(\alpha,c) = v_\infty(\alpha,c) 
= 0$ by Lemma~\ref{lem:limits}, where $\alpha = \alpha_1(c)$, hence we
can take $\bar y \in \R$ large enough so that $u_{\alpha,c}(\bar y) + 
d v_{\alpha,c}(\bar y) < 1$. By continuity, we then have $u_{\alpha',c}(\bar y) +
d v_{\alpha',c}(\bar y) < 1$ for any $\alpha'$ sufficiently close to $\alpha$, so 
that $u_\infty(\alpha',c) = 0$ for some $\alpha' > \alpha_1(c)$, in contradiction 
with the definition of $\alpha_1(c)$. It is more difficult to prove
that $v_\infty(\alpha_1(c),c) = 0$; this is precisely the purpose
of the next section. 
\end{rem}

\subsection{The center manifold of the healthy state}
\label{subsec25}

Given any $c > 0$, we assume from now on that $\alpha \ge \alpha_1(c)$. 
In that case, we know from Remark~\ref{rem:alpha1} that $T(\alpha,c) = 
+\infty$ and that the solution  $(u_{\alpha,c},v_{\alpha,c}, w_{\alpha,c})$ 
of \eqref{Ode3} given by Lemma~\ref{lem:unstab} converges to 
$\bigl(1,v_\infty(\alpha,c),0\bigr)$ as $y \to + \infty$, 
where $0 \le v_\infty(\alpha,c) < 1$. Our goal is to determine
for which value(s) of $\alpha$ we have $v_\infty(\alpha,c) = 0$, 
so that the boundary conditions \eqref{BC3} are satisfied. 

To study the dynamics of the ODE system \eqref{Ode3} in a 
neighborhood of the healthy equilibrium $S_+ = (1,0,0)$, 
we introduce the new dependent variables 
\begin{equation}\label{tildeuvw}
  \tilde u(y) \,=\, 1 - u(y)\,, \qquad \tilde v(y) 
  \,=\, v(y) + w(y)/c\,, \qquad \tilde w(y) = w(y)/c\,,
\end{equation}
which satisfy the modified system
\begin{align}\nonumber
  \tilde u' \,&=\,  \frac{1}{c}\,\tilde u(1-\tilde u)(\tilde u - d\tilde v
    + d\tilde w)\,, \\ \label{Ode4}
  \tilde v' \,&=\, -\frac{r}{c}\,\tilde u(\tilde v -\tilde w)(1-
    \tilde v + \tilde w)\,, \\ \nonumber
  \tilde w' \,&=\, -c \tilde w - \frac{r}{c}\,\tilde u(\tilde v -\tilde w)
  (1-\tilde v + \tilde w)\,.
\end{align}
It is clear that $(\tilde u,\tilde v,\tilde w) = (0,v_\infty,0)$ is an equilibrium of
\eqref{Ode4} for any $v_\infty \in \R$, and that the healthy state $S_+$ corresponds to
$v_\infty = 0$. The linearization of \eqref{Ode4} at the origin is easily found to be
$\tilde u' = 0$, $\tilde v' = 0$, $\tilde w' = -c\tilde w$. It follows that all solutions
of \eqref{Ode4} that stay in a small neighborhood of the origin for all sufficiently large
$y > 0$ converge as $y \to +\infty$ to a two-dimensional center manifold
$\cW \subset \R^3$, which is tangent at the origin to the subspace spanned by the vectors
$(1,0,0)$ and $(0,1,0)$. For any $k \in \N$, the center manifold is locally the graph of a
$C^k$ function $\cF$, so that $(\tilde u,\tilde v,\tilde w) \in \cW$ if and only if
$\tilde w = \cF(\tilde u,\tilde v)$ where
\begin{equation}\label{cFdef}
  \cF(\tilde u,\tilde v) \,=\,
  -\frac{r}{c^2}\,\tilde u\tilde v \Bigl(1 +
  \cO(|\tilde u| + |\tilde v|)\Bigr)\,,
  \qquad \hbox{as} \quad (\tilde u,\tilde v) \to
  (0,0)\,.
\end{equation}
We recall that the center manifold $\cW$ and the associated function $\cF$ are not
necessarily unique, but the asymptotic expansion in \eqref{cFdef} is free of ambiguity,
see e.g. \cite[Chapter~3]{GH}.  Moreover, any center manifold $\cW$ necessarily contains
the equilibria $(0,v_\infty,0)$ for sufficiently small values of $v_\infty$. The
derivation of \eqref{cFdef} is standard, see \cite{CLW,GH,HI} for the methodology and
several examples. We just observe here that $\cF(0,\tilde v) = 0$ because the dynamics of
\eqref{Ode4} is trivial when $\tilde u = 0$, and that $\cF(\tilde u,0) = 0$ because the
subspace defined by $\tilde v = \tilde w = 0$ is invariant under the evolution defined by
\eqref{Ode4}.

Since we are interested in solutions of \eqref{Ode3} that stay in the region $\cD$ defined
by \eqref{Ddef}, it is natural to consider solutions of \eqref{Ode4} on the smaller
manifold
\[
  \cW_+ \,=\, \bigl\{(\tilde u,\tilde v,\tilde w) \in 
  \cW \,\big|\, \tilde u > 0\,,~\tilde v > 0\bigr\}\,.
\]
We first study the solutions of \eqref{Ode4} which
converge to zero as $y \to +\infty$. 

\begin{lem}\label{lem:Wzero}
Up to translations in the variable $y$, there exists a unique 
solution of \eqref{Ode4} on the center manifold $\cW_+$ which 
converges to zero as $y \to +\infty$. This solution 
satisfies
\begin{equation}\label{asym2}
  \tilde u(y) \,=\, \frac{c}{ry} + \cO\Bigl(\frac{1}{y^2}
  \Bigr)\,, \qquad 
  \tilde v(y) \,=\, \frac{c(1+r)}{dry} + \cO\Bigl(\frac{1}{y^2}
  \Bigr)\,, \qquad \hbox{as } y \to +\infty\,.
\end{equation}
\end{lem}

\begin{proof}
We first prove the existence of a solution of \eqref{Ode4} on $\cW_+$ 
which converges to zero as $y \to +\infty$. We perform the 
change of variables
\begin{equation}\label{fgdef}
  \tilde u(y) \,=\, \frac{a}{y}\,f\bigl(\ln y\bigr)\,, \qquad
  \tilde v(y) \,=\, \frac{b}{y}\,g\bigl(\ln y\bigr)\,, \qquad
  z \,=\, \ln y\,,
\end{equation}
where $a = c/r$ and $b = c(1+r)/(dr)$. If $\tilde u,\tilde v$ 
evolve according to \eqref{Ode4} with $\tilde w = \cF(\tilde u,\tilde v)$, 
the new functions $f(z),g(z)$ satisfy the system
\begin{equation}\label{fgsys}
\begin{aligned}
  f' \,&=\, f + \frac{1}{c}f\Bigl(1 - a e^{-z}f\Bigr)
  \Bigl(af -db g + d e^{-z} \cR(f,g,z)\Bigr)\,, \\
  g' \,&=\, g - \frac{1}{b}f\Bigl(bg - e^{-z}\cR(f,g,z)\Bigr)
  \Bigl(1 - b e^{-z}g + e^{-2z}\cR(f,g,z)\Bigr)\,,
\end{aligned}
\end{equation}
where $\cR(f,g,z) = e^{2z}\cF(a e^{-z}f,be^{-z}g)$ and ${}'$ now
denotes differentiation with respect to the new variable 
$z = \ln y$. As $|\cR(f,g,z)| \le C |f| |g|$ by \eqref{cFdef},
we see that the non-autonomous system \eqref{fgsys} converges 
as $z \to +\infty$ to
\begin{equation}\label{fglimit}
  f' \,=\, f \left\{1 + \frac{1}{r}f - \left(1 + \frac{1}{r}\right) g\right\}\,, 
  \qquad g' \,=\, g\bigl(1 - f\bigr)\,.
\end{equation}
This limiting system has a unique positive equilibrium $(\bar f,\bar g) = (1,1)$, which is
hyperbolic, and the eigenvalues of the linearized operator are easily found to be
$1 + 1/r$ and $-1$. Applying the stable manifold theorem, we deduce that there exists a
solution $(f,g)$ of \eqref{fgsys} which converges to $(1,1)$ as $z \to +\infty$ and
satisfies $|f(z) - 1| + |g(z)-1| = \cO(e^{-z})$ in this limit. Returning to the original
variables, we conclude that the solution of \eqref{Ode4} on $\cW_+$ given by \eqref{fgdef}
converges to $(0,0)$ as $y \to +\infty$ and satisfies \eqref{asym2}.

To prove uniqueness, it is convenient to write the evolution 
equations on the center manifold $\cW_+$ in the condensed form 
\begin{equation}\label{uvreduc}
  \tilde u' \,=\, \cG(\tilde u,\tilde v)\,, \qquad 
  \tilde v' \,=\, \cH(\tilde u,\tilde v)\,,
\end{equation}
where 
\begin{equation}\label{GHdef}
\begin{aligned}
  \cG(\tilde u,\tilde v) \,&=\, \frac{1}{c}\,\tilde u(1-\tilde u)
  \bigl(\tilde u - d\tilde v + d\cF(\tilde u,\tilde v)\bigr)\,, \\
  \cH(\tilde u,\tilde v) \,&=\, -\frac{r}{c}\,\tilde u\bigl(\tilde v -
  \cF(\tilde u,\tilde v)\bigr)\bigr(1- \tilde v + \cF(\tilde u,
  \tilde v)\bigr)\,. 
\end{aligned}
\end{equation}
The solution $(\tilde u,\tilde v)$ of \eqref{uvreduc} constructed in 
the previous step satisfies $\tilde v = \Psi(\tilde u)$ in some 
$\epsilon$-neighborhood of the origin, where $\Psi : (0,\epsilon) 
\to \R_+$ is a $C^k$ function satisfying the functional relation
\begin{equation}\label{funcrel}
  \cH\bigl(x,\Psi(x)\bigr) \,=\, \Psi'(x)\,\cG\bigl(x,\Psi(x)\bigr)\,, 
  \qquad x \in (0,\epsilon)\,.
\end{equation}
Moreover, in agreement with \eqref{asym2}, we have $\Psi(x) = 
(1+r)x/d + \cO(x^2)$ as $x \to 0$. 

Now, we consider an {\em arbitrary} positive solution $(\tilde u,\tilde v)$ 
of \eqref{uvreduc} that converges to the origin as $y \to +\infty$.
Using \eqref{funcrel}, we observe that
\begin{align}\nonumber
  \frac{\D}{\D y}\bigl(\tilde v - \Psi(\tilde u)\bigr) \,&=\, 
  \cH\bigl(\tilde u,\tilde v\bigr) - \Psi'(\tilde u)\,
  \cG\bigl(\tilde u,\tilde v\bigr) \\[-1mm] \label{diffeq}
  \,&=\, \cH\bigl(\tilde u,\tilde v\bigr) - \cH\bigl(\tilde u,
  \Psi(\tilde u)\bigr) - \Psi'(\tilde u)\Bigl(\cG\bigl(\tilde u,
  \tilde v\bigr) - \cG\bigl(\tilde u,\Psi(\tilde u)\bigr)\Bigr) \\
  \nonumber
  \,&=\, \Delta(\tilde u,\tilde v)\bigl(\tilde v - \Psi(\tilde u)\bigr)\,,
\end{align}
where
\[
  \Delta(\tilde u,\tilde v) \,=\, \int_0^1 \Bigl(\partial_2 
  \cH\bigl(\tilde u,(1-t)\Psi(\tilde u)+t \tilde v\bigr) - 
  \Psi'(\tilde u)\,\partial_2 \cG\bigl(\tilde u,(1-t)\Psi(\tilde u)
  +t \tilde v\bigr)\Bigr)\dd t\,. 
\]
Using \eqref{cFdef} and \eqref{GHdef}, it is straightforward to 
compute an asymptotic expansion of $\Delta(\tilde u,\tilde v)$ 
as $(\tilde u,\tilde v) \to (0,0)$, which is found to be 
$\Delta(\tilde u,\tilde v) = c^{-1}\tilde u \bigl(1 + \cO(|\tilde u| +
|\tilde v|)\bigr)$. In particular $\Delta(\tilde u,\tilde v) > 0$ for 
small solutions on $\cW_+$. Keeping that observation in mind, 
we integrate \eqref{diffeq} on the interval $[y_1,y_2]$ for $y_1 > 0$ 
sufficiently large and obtain the relation
\[
  \tilde v(y_2) - \Psi(\tilde u(y_2)) \,=\, \exp\Bigl(
  \int_{y_1}^{y_2} \Delta(\tilde u(y),\tilde v(y))\dd y\Bigr)
  \Bigl(\tilde v(y_1) - \Psi(\tilde u(y_1))\Bigr)\,,
\]
which implies that $|\tilde v(y_2) - \Psi(\tilde u(y_2))| \ge 
|\tilde v(y_1) - \Psi(\tilde u(y_1))|$. By assumption, the 
left-hand side converges to zero as $y_2 \to +\infty$, 
and we conclude that $\tilde v(y_1) = \Psi(\tilde u(y_1))$ for all 
(sufficiently large) $y_1 > 0$. This precisely means that 
$(\tilde u,\tilde v)$ coincides, up to a translation in the 
variable $y$, with the solution of \eqref{uvreduc} constructed 
in the first step. 
\end{proof}

Since a whole neighborhood of the origin in $\R^3$ is foliated by
one-dimensional strong stable leaves over the two-dimensional center
manifold $\cW$, see \cite{CLW}, we can extract from Lemma~\ref{lem:Wzero}
useful information on the asymptotic behavior as $y \to +\infty$
of the traveling waves of the original system \eqref{Ode3}.
As a first application, we prove uniqueness of the traveling wave
for each value of the speed parameter. 

\begin{lem}\label{lem:unique}
Given any $c > 0$, there exists at most one value $\alpha \ge
\alpha_1(c)$ of the shooting parameter such that the solution
$(u_{\alpha,c}(y),v_{\alpha,c}(y),w_{\alpha,c}(y))$ of \eqref{Ode3} given
by Lemma~\ref{lem:unstab} converges to $S_+ = (1,0,0)$ as $y \to +\infty$. 
\end{lem}

\begin{proof}
Assume that, for some $\alpha \ge \alpha_1(c)$, the solution
$(u_{\alpha,c},v_{\alpha,c},w_{\alpha,c})$ of \eqref{Ode3}
converges to $S_+ = (1,0,0)$ as $y \to +\infty$. We denote by
$(\tilde u_{\alpha,c},\tilde v_{\alpha,c},\tilde w_{\alpha,c})$
the corresponding solution of \eqref{Ode4}, given by the
change of variables \eqref{tildeuvw}. We first observe that
$(\tilde u_{\alpha,c},\tilde v_{\alpha,c},\tilde w_{\alpha,c})$
does not lie on the strong stable manifold of the origin
$(0,0,0)$, because that manifold consists of solutions
of \eqref{Ode4} satisfying $\tilde u' = \tilde v' = 0$, $\tilde
w' = -c \tilde w$. Thus $(\tilde u_{\alpha,c},\tilde v_{\alpha,c},
\tilde w_{\alpha,c})$ approaches exponentially fast a nontrivial
solution $(\tilde u,\tilde v)$ on the center manifold $\cW_+$,
which converges itself to $(0,0)$ as $y \to +\infty$. 
Using Lemma~\ref{lem:Wzero}, we conclude that the pair
$(\tilde u_{\alpha,c},\tilde v_{\alpha,c})$ satisfies the
asymptotic expansion \eqref{asym2}, and that $\tilde
w_{\alpha,c} = \cF(\tilde u_{\alpha,c},\tilde v_{\alpha,c})$
up to exponentially small corrections as $y \to +\infty$. 

Now suppose that another solution $(u_{\alpha'\!,c},v_{\alpha'\!,c}, w_{\alpha'\!,c})$
also converges to $S_+$. If $\alpha' > \alpha$, the proof of Lemma~\ref{lem:monotone}
shows that there exists $\rho > 1$ such that $v_{\alpha'\!,c}(y)/v_{\alpha,c}(y) \ge \rho$
for all sufficiently large $y > 0$. On the other hand, since the asymptotic behavior of
both solutions is given by \eqref{asym2}, as is explained above, it follows from
\eqref{tildeuvw} and \eqref{cFdef} that
\[
  \lim_{y \to +\infty} \frac{v_{\alpha'\!,c}(y)}{v_{\alpha,c}(y)}
  \,=\, \lim_{y \to +\infty} \frac{\tilde v_{\alpha'\!,c}(y)}{\tilde
  v_{\alpha,c}(y)} \,=\, 1\,,
\]
which gives a contradiction. So we must have $\alpha' = \alpha$
and uniqueness is established. 
\end{proof}

We now prove the main result of this section, namely the 
existence of a traveling wave connecting the infected
state $S_-$ to the healthy state $S_+$. 

\begin{lem}\label{lem:exist}
If $c > 0$ and $\alpha = \alpha_1(c)$, the solution $(u_{\alpha,c},
v_{\alpha,c},w_{\alpha,c})$ of \eqref{Ode3} given by
Lemma~\ref{lem:monotone} converges to $S_+ = (1,0,0)$ as
$y \to +\infty$. 
\end{lem}

\begin{proof}
Fix $\alpha = \alpha_1(c)$.  We already know that
$(u_{\alpha,c}(y),v_{\alpha,c}(y),w_{\alpha,c}(y))$ converges to
$(1,v_\infty,0)$ as $y \to +\infty$, for some $v_\infty \in [0,1)$. If
$v_\infty > 0$, we obtain a contradiction as follows. We perform
again the change of variables \eqref{tildeuvw} and consider system
\eqref{Ode4} near the equilibrium $(0,v_\infty,0)$.  The
linearization at this point is given by
\[
  \tilde u' \,=\, -\frac{dv_\infty}{c}\,\tilde u\,, \qquad
  \tilde v' \,=\, -\frac{rv_\infty}{c}(1-v_\infty)\,\tilde u\,, \qquad
  \tilde w' \,=\, -c \tilde w -\frac{rv_\infty}{c}(1-v_\infty)\,\tilde u
  \,.
\]
In contrast to the situation where $v_\infty = 0$, which was studied
previously, the zero eigenvalue is now simple, with eigenvector
$(0,1,0)$, and there are two negative eigenvalues $-dv_\infty/c$ and
$-c$. Applying the center manifold theorem again, we deduce that there
exists a small open neighborhood $\Omega$ of $(0,v_\infty,0)$ in
$\R^3$ that is foliated by two-dimensional stable leaves over a
one-dimensional center manifold, which itself consists of the family
of equilibria $(0,v,0)$ with $v$ close enough to $v_\infty$. Taking a
smaller neighborhood if needed, we can make sure that, for all initial
data in $\Omega$, the solution of \eqref{Ode4} converges to $(0,\bar
v,0)$ as $y \to +\infty$, for some $\bar v > 0$.

Now let $(\tilde u_{\alpha,c},\tilde v_{\alpha,c},\tilde w_{\alpha,c})$
denote the solution of \eqref{Ode4} obtained from $(u_{\alpha,c},
v_{\alpha,c},w_{\alpha,c})$ by the change of variables \eqref{tildeuvw}.
By assumption $(\tilde u_{\alpha,c}(y),\tilde v_{\alpha,c}(y),\tilde
w_{\alpha,c}(y))$ converges to $(0,v_\infty,0)$ as $y \to +\infty$, hence
there exists $\bar y \in \R$ such that $(\tilde u_{\alpha,c}(y),\tilde
v_{\alpha,c}(y), \tilde w_{\alpha,c}(y)) \in \Omega$ for all $y \ge \bar y$.
By continuity, we infer that $(\tilde u_{\alpha'\!,c}(\bar y),\tilde
v_{\alpha'\!,c}(\bar y), \tilde w_{\alpha'\!,c}(\bar y)) \in \Omega$ if
$\alpha' < \alpha$ is sufficiently close to $\alpha$, which means
that $(u_{\alpha'\!,c}(y),v_{\alpha'\!,c}(y),w_{\alpha'\!,c}(y))$ converges
to $(1,\bar v,0)$ as $y \to +\infty$ for some $\bar v > 0$. As $\alpha' <
\alpha = \alpha_1(c)$, this clearly contradicts definition \eqref{alpha1def}.
So we must have $v_\infty(\alpha,c) = 0$. 
\end{proof}

\begin{rem}\label{rem:monotone}
Using similar arguments, one can also show that the map
$\alpha \mapsto v_\infty(\alpha,c)$ is continuous and strictly
increasing for $\alpha \ge \alpha_1(c)$. 
\end{rem}

\subsection{Asymptotic behavior in the original variables}
\label{subsec26}

It is now an easy task to complete to proof of Theorem~\ref{main1}. Given any 
$c > 0$, we denote $\alpha = \alpha_1(c) > 0$, where $\alpha_1(c)$ is defined 
in \eqref{alpha1def}. We know from Lemma~\ref{lem:exist} that the solution $(u,v,w) =
(u_{\alpha,c},v_{\alpha,c}, w_{\alpha,c})$ of \eqref{Ode3} given by Lemma~\ref{lem:unstab} 
converges to $S_+ = (1,0,0)$ as $y \to +\infty$, so that the boundary conditions 
\eqref{BC3} are satisfied. We now undo the change of variables \eqref{ydef}, which in
view of \eqref{uvdef} can be written in the equivalent form
\begin{equation}\label{xidef}
  \frac{\D \xi}{\D y} \,\equiv\, \bigl(\Phi^{-1}\bigr)'(y)
  \,=\, 1 - u(y)\,, \qquad y \in \R\,.
\end{equation}
Specifically, we define
\begin{equation}\label{xiexp}
  \xi(y) \,=\, \Phi^{-1}(y) \,=\, y - \int_{-\infty}^y
  u(y')\dd y'\,, \qquad y \in \R\,.
\end{equation}
Using the asymptotic expansions \eqref{asym1} as $y \to -\infty$ and \eqref{asym2} 
as $y \to +\infty$, it is straightforward to verify that
\begin{equation}\label{Phiminus}
  \Phi^{-1}(y) \,=\, \begin{cases}
    y - \frac{\alpha}{\mu}\,e^{\mu y} + \cO\Bigl(e^{(\mu+\eta)y}\Bigr)
    & \hbox{as} \quad y \to -\infty\,, \\
    \frac{c}{r}\,\ln(y) + \xi_0 + \cO\Bigl(\frac{1}{y}\Bigr)
    & \hbox{as} \quad y \to +\infty\,, \\
    \end{cases}
\end{equation}
for some $\xi_0 \in \R$. At this point, it is important to note that $\xi(y) 
\to +\infty$ as $y \to +\infty$, so that $\bar\xi = +\infty$ in the terminology
of Lemma~\ref{lem:sharp}. Sharp fronts of the original system \eqref{Ode1}
would correspond to solutions of \eqref{Ode3} satisfying $\int_0^{+\infty}
(1{-}u)\dd y < \infty$, which are excluded by Lemma~\ref{lem:Wzero}. Inverting 
\eqref{Phiminus}, we easily find
\begin{equation}\label{Phiplus}
  \Phi(\xi) \,=\, \begin{cases}
    \xi + \frac{\alpha}{\mu}\,e^{\mu \xi} + \cO\Bigl(e^{(\mu+\eta)\xi}\Bigr)
    & \hbox{as} \quad \xi \to -\infty\,, \\
    e^{\gamma(\xi-\xi_0)} + \cO(1)
    & \hbox{as} \quad \xi \to +\infty\,, \\
    \end{cases}
\end{equation}
where $\gamma = r/c$. Finally, defining $\cU(\xi) = u(\Phi(\xi))$
and $\cV(\xi) = v(\Phi(\xi))$ in agreement with \eqref{uvdef},
we obtain by construction a solution of \eqref{Ode1} which
satisfies the boundary conditions \eqref{BC1}. Since $u'(y) < 0$
and $v'(y) > 0$ for all $y \in \R$, it is clear that $\cU'(\xi) < 0$
and $\cV'(\xi) > 0$ for all $\xi \in \R$, and the asymptotic
expansions \eqref{AsymUV1}, \eqref{AsymUV2} are direct consequences
of \eqref{asym1}, \eqref{asym2}, and \eqref{Phiplus}. For any
$c > 0$, the uniqueness (up to translations) of the solution of
\eqref{Ode1} satisfying \eqref{BC1} is a consequence of
Lemma~\ref{lem:unique}. The proof of Theorem~\ref{main1}
is thus complete. \QED

\begin{figure}[ht]
  \begin{center}
  \begin{picture}(480,160)
  \put(20,0){\includegraphics[width=0.43\textwidth]{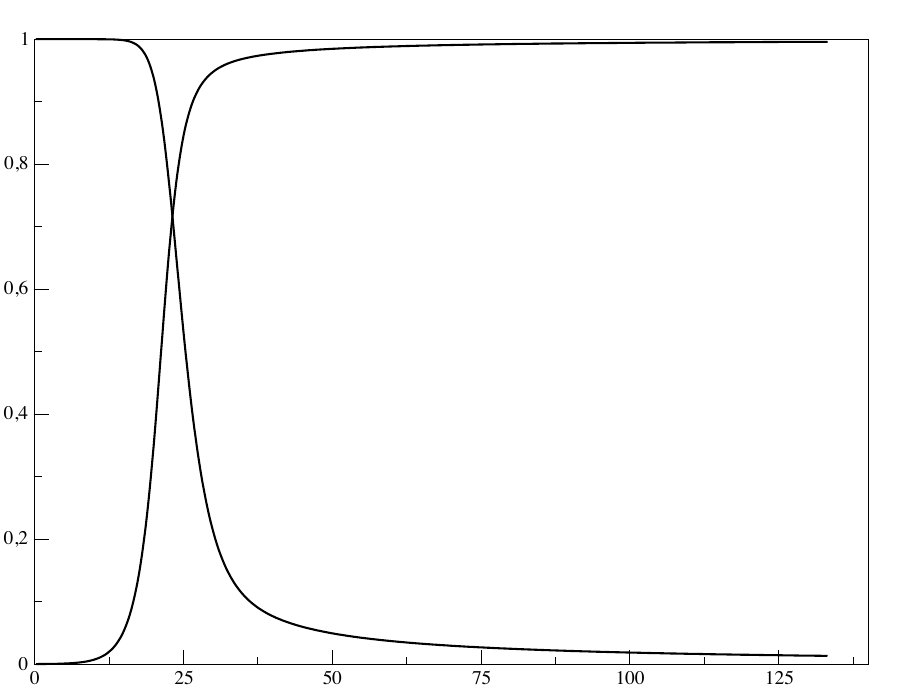}}
  \put(250,0){\includegraphics[width=0.43\textwidth]{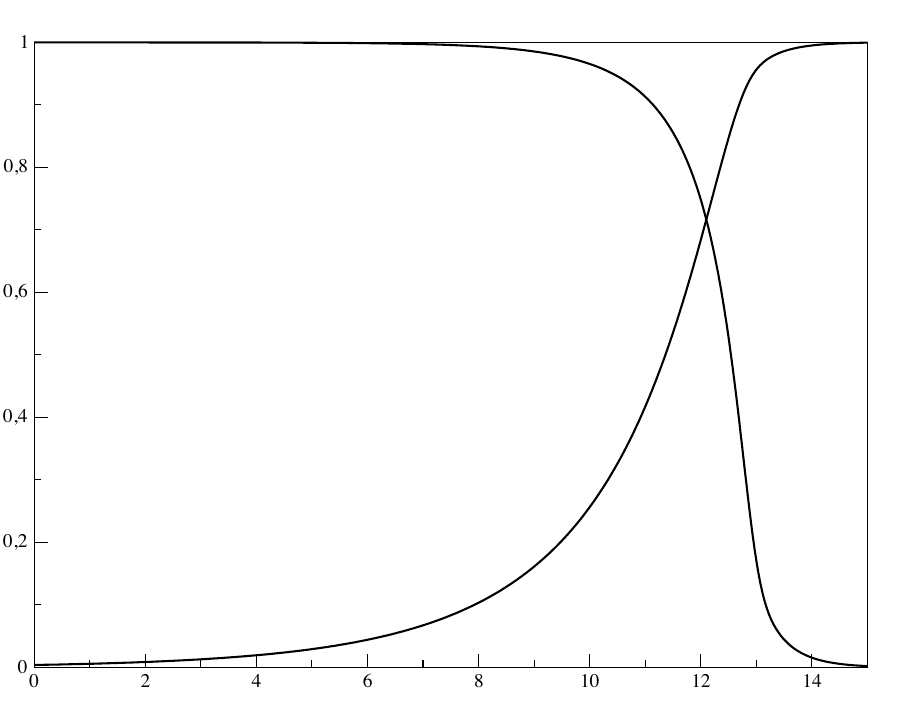}}
  \put(32,66){$u(y)$}
  \put(64,66){$v(y)$}
  \put(364,66){$\cU(\xi)$}
  \put(412,66){$\cV(\xi)$}
  \put(65,100){$\xrightarrow[\hspace{1cm}]{}$}
  \put(405,100){$\xrightarrow[\hspace{0.6cm}]{}$}
  \put(120,32){{\small desingularized}}
  \put(137,20){{\small variables}}
  \put(265,20){{\small original variables}}
  \end{picture}
  \caption{{\small The profile of the same propagation front is represented in
  the desingularized variables (left) and in the original variables (right). 
  The values of the parameters are $d = 2$, $r = 5$, and $c = 0.5$. 
  Note that $1-u(y)$ and $v(y)$ converge slowly to zero as $y \to +\infty$, 
  in agreement with \eqref{asym2}, whereas $1-\cU(\xi)$ and $\cV(\xi)$ 
  decay exponentially and may even exhibit a sharp edge when $c$ is small 
  enough. In contrast, the behavior near $-\infty$ is identical in both sets 
  of variables, which is not obvious here because the horizontal scales are 
  very different.} \label{fig3}}
  \end{center}
\end{figure}
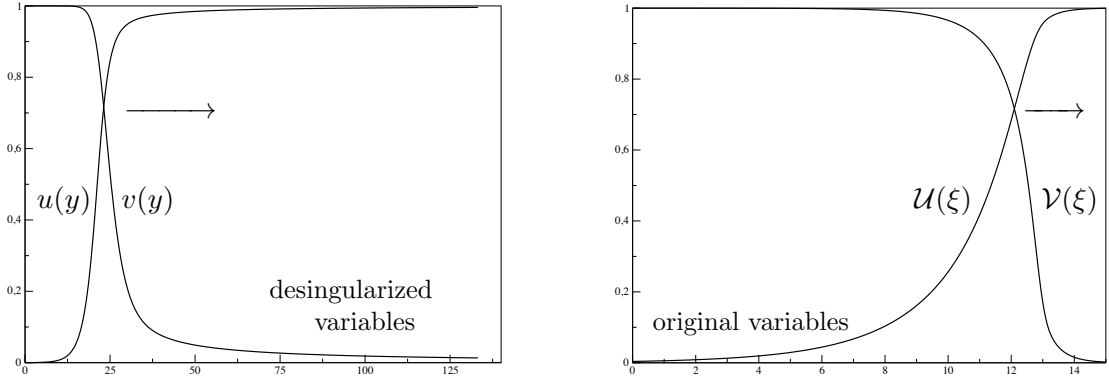

\subsection{Existence of traveling waves when $d < 1$}
\label{subsec27}

Since the beginning of Section~\ref{subsec22}, we assumed that the
parameter $d$ in \eqref{redGGsys} is larger than one, which seems to be
the most relevant situation in cancerology, see \cite{McGiEtAl14}.
For completeness, we now consider the opposite case where $0 < d < 1$.
The analysis being very similar, we just indicate how the proof of
Theorem~\ref{main1} can be modified to obtain the conclusions of
Theorem~\ref{main2}.

Our starting point is again the desingularized system \eqref{Ode3},
which has now the following nontrivial equilibria\: the infected state
$S_- = (0,1,0)$, the healthy state $S_+ = (1,0,0)$, the coexistence
state $S_d = (1-d,1,0)$, and the artificial equilibria
$(1,v_\infty,0)$ where $v_\infty \neq 0$.  The first important
observation is that there exists no traveling wave connecting $S_-$ to
$S_+$ in that case.  Indeed, it is clear from the linearization
\eqref{linS-} that the unstable manifold of the infected state $S_-$
is one-dimensional when $d < 1$.  Solutions on that manifold are of
the form $(0,v,v')$, where $v$ solves the Fisher--KPP equation
\eqref{FKPP} and $v(y) \to 1$ as $y \to -\infty$.
Since $u$ is equal to  $0$ on the unstable manifold, we never obtain a
heteroclinic connection between $S_-$ and $S_+$. 

We now consider solutions on the unstable manifold of the 
coexistence state $S_d = (1-d,1,0)$. Linearizing \eqref{Ode3}
at $S_d$, we obtain
\begin{equation}\label{linSd}
  \tilde u' \,=\, \frac{d(1-d)}{c}\,\bigl(\tilde u - d\tilde v\bigr)\,, 
  \qquad \tilde v ' \,=\, - w\,, \qquad w' \,=\, -cw -dr\tilde v\,,
\end{equation}
where $\tilde u = u - 1 + d$ and $\tilde v = 1 - v$. We thus find
two positive eigenvalues
\begin{equation}\label{lamudef2}
  \lambda \,=\, \frac12\bigl(-c + \sqrt{c^2 + 4dr}\bigr) \,>\, 0\,,
  \qquad \mu \,=\, \frac{d(1-d)}{c} \,>\, 0\,,
\end{equation}
as well as one negative eigenvalue $-\frac12\bigl(c + \sqrt{c^2 + 4dr}
\bigr) < 0$. So we again have a two-dimensional unstable manifold, 
and the analogue of Lemma~\ref{lem:unstab} is\:

\begin{lem}\label{lem:unstab2} For any $\alpha \in \R$, 
the ODE system \eqref{Ode3} has a unique solution such that
\begin{equation}\label{asym3}
\begin{aligned}
  u(y) \,&=\, 1 - d + \alpha\,e^{\mu y} + \frac{d\mu}{\mu-\lambda}
  \Bigl(e^{\lambda y} - e^{\mu y}\Bigr) + \cO\Bigl(e^{(\mu + \eta)y}
  \Bigr)\,, \\ 
  v(y) \,&=\, 1 - e^{\lambda y} + \cO\Bigl(e^{(\lambda + \eta)y}\Bigr)\,, 
  \quad
  w(y) \,=\, - \lambda\,e^{\lambda y} + \cO\Bigl(e^{(\lambda + \eta)y}\Bigr)\,,
\end{aligned}
\end{equation}
as $y \to -\infty$, where $\lambda,\mu$ are given by \eqref{lamudef2}
and $\eta = \min(\lambda,\mu) > 0$. 
\end{lem}

Of course, in the particular case where $\lambda = \mu$, the first 
equation in \eqref{asym3} should read
\[
  u(y) \,=\, 1 - d + \bigl(\alpha - d\mu y\bigr)\,e^{\mu y} 
  + \cO\Bigl(e^{2\mu y}\Bigr)\,, \qquad \hbox{as }y \to -\infty\,.
\]
The main difference with Lemma~\ref{lem:unstab} is that the shooting parameter $\alpha$
can take arbitrary values in $\R$, and is not requested to be positive. The reason is that
we look for solutions of \eqref{Ode3} that lie in the region $\cD$ defined by
\eqref{Ddef}, which is the case of all solutions \eqref{asym3} in the asymptotic regime
$y \to -\infty$, even if $\alpha < 0$.

As in the proof of Theorem~\ref{main1}, the strategy is to find an appropriate value of
the shooting parameter $\alpha \in \R$ so that the solution of \eqref{Ode3} defined by
\eqref{asym3} converges to $S_+$ as $y \to +\infty$.  We first observe that
Lemma~\ref{lem:invariance} still holds, so that we can define $T(\alpha,c)$ by
\eqref{Tdef} for all $\alpha \in \R$. Next, as in Lemma~\ref{lem:monotone}, we claim that
$T(\alpha,c)$ is an increasing function of $\alpha$ and that inequalities \eqref{comp1}
hold when $\alpha_2 > \alpha_1$ and $y \in (-\infty,T(\alpha_1,c))$. The first part of the
proof of Lemma~\ref{lem:monotone} uses a Taylor approximation of system \eqref{Ode3} near
$S_-$ and must therefore be modified since the starting point is now the coexistence state
$S_d \neq S_-$. It is clear from \eqref{asym3} that $u_{\alpha_2,c}(y) > u_{\alpha_1,c}(y)$ 
when $y < 0$ is sufficiently large, and straightforward calculations show that the 
expression \eqref{omexp} of $\omega(y) = e^{-\lambda y}\bigl( v_{\alpha_2,c}(y) - 
v_{\alpha_1,c}(y)\bigr)$ has to be replaced by
\begin{equation}\label{omexp2}
  \omega(y) \,=\, \frac{r(\alpha_2 - \alpha_1)}{\mu(\mu+
  \sqrt{\mu^2 + 4rd})}\,e^{\mu y} + \cO\Bigl(e^{(\mu + \eta)y}\Bigr)\,, 
  \qquad \hbox{as }y \to -\infty\,.
\end{equation}
We deduce as before that inequalities \eqref{comp1} holds when $y < 0$ is large
 enough, and the second part of the proof is unchanged. 

As in \eqref{alpha0def}, we define
\begin{equation}\label{alpha0def2}
  \alpha_0(c) \,=\, \inf\bigl\{\alpha \in \R \,|\, T(\alpha,c) = +\infty
  \bigr\} \,\in\, [-\infty,+\infty]\,,
\end{equation}
and we have the following analogue of Lemma~\ref{lem:alpha0}\:

\begin{lem}\label{lem:alpha02}
If $c \ge 2\sqrt{r}$, then $\alpha_0(c) = -\infty$. If $0 < c < 2\sqrt{r}$, 
then $-\infty < \alpha_0(c) < +\infty$. 
\end{lem}

\begin{proof}
If $(u_{\alpha,c},v_{\alpha,c},w_{\alpha,c})$ denotes the solution of \eqref{Ode3} 
satisfying \eqref{asym3}, and $v$ is the solution of the Fisher--KPP equation 
\eqref{FKPP} normalized so that $e^{-\lambda y}(1-v(y)) \to 1$ as $y \to -\infty$, 
we first observe that $v_{\alpha,c}(y) > v(y)$ as long as $v(y) > 0$, because 
\[
  v_{\alpha,c}'' + cv_{\alpha,c}' + rv_{\alpha,c}(1-v_{\alpha,c}) 
  \,>\, v_{\alpha,c}'' + cv_{\alpha,c}' + rv_{\alpha,c}(1-u_{\alpha,c})
  (1-v_{\alpha,c}) \,=\, 0\,. 
\]
If $c \ge 2\sqrt{r}$, we know that the Fisher--KPP front $v$ remains positive, and this
implies that $T(\alpha,c) = + \infty$ for all $\alpha \in \R$.

In the weakly damped case where $0 < c < 2\sqrt{r}$, we use the same arguments as in the
proof of Lemma~\ref{lem:alpha0}, with suitable modifications. First, if $\alpha < 0$ is
large enough, one can prove verify \eqref{asym3} that the quantity $u(y) + d v(y) - 1$
takes negative values for some (large) $y < 0$. This is obvious when $\mu < \lambda$,
because any sufficiently large $y < 0$ has the desired property, but if $\mu > \lambda$
one has to choose $y$ such that $-(\log|\alpha|)/(\mu{-}\lambda) \ll y \ll 
-(\log|\alpha|)/\mu$ (the details being left to the reader). Therefore, assuming that 
$T(\alpha,c) = +\infty$, we deduce as in the proof of Lemma~\ref{lem:limits} that 
$u_{\alpha,c}(y) \to 0$ as $y \to +\infty$.  This in turn implies, as in the proof of
Lemma~\ref{lem:alpha1}, that $v_{\alpha,c}(y)$ satisfies a weakly damped Fisher--KPP
equation for large $y$, and must therefore change sign, which gives a contradiction.  So
$T(\alpha,c) < \infty$ if $\alpha < 0$ is sufficiently large.

Finally, if $\alpha > 0$ is large enough, we prove as in Lemma~\ref{lem:alpha0} that the
function $u_{\alpha,c}$ converges to $1$ so rapidly that the non-linearity in the equation
for $v_{\alpha,c}$ becomes totally depleted before $v_{\alpha,c}(y)$ leaves a small
neighborhood of the initial point. Thus $v_{\alpha,c}$ remains close to a solution of the
linear equation $v'' + cv' = 0$, hence converges to a nonzero limit $v_\infty$ as
$y \to +\infty$.  In particular, we have $T(\alpha,c) = +\infty$ if $\alpha > 0$ is large
enough. We leave the details to the reader.
\end{proof}

The rest of the proof of Theorem~\ref{main2} follows the arguments given in
Sections~\ref{subsec24}--\ref{subsec26} without substantial modifications. In particular,
Lemma~\ref{lem:limits} is unchanged, so that we can define $\alpha_1(c)$ as in
\eqref{alpha1def}, and the analogue of Lemma~\ref{lem:alpha1} asserts that
$-\infty < \alpha_1(c) < +\infty$ for any $c > 0$. If $\alpha = \alpha_1(c)$, we have
$T(\alpha,c) = +\infty$ and $u_\infty(\alpha,c) = 1$ as in Remark~\ref{rem:alpha1}, and
finally $v_\infty(\alpha,c) = 0$ because the analysis on the center manifold of the healthy
state, which is given in Section~\ref{subsec25}, does not depend on the value of the
parameter $d > 1$. This proves the existence of a (unique) heteroclinic trajectory of
system~\eqref{Ode3} connecting the coexistence state $S_d$ to the healthy state $S_+$,
when $0 < d < 1$.

When returning to the original variables, we have to keep 
in mind that the change of variables \eqref{xidef} is not 
close to identity for large negative values of $y$, because 
$u(y) \to 1-d$ as $y \to -\infty$. Instead of \eqref{xiexp}, we 
thus define
\[
  \xi(y) \,=\, \Phi^{-1}(y) \,=\, dy - \int_{-\infty}^y
  \Bigl(u(y') + d -1\Bigr)\dd y'\,, \qquad y \in \R\,,
\]
and we observe that the function $\Phi$ satisfies\:
\[
  \Phi(\xi) \,=\, \begin{cases}
    \frac{1}{d}\,\xi + \cO\Bigl(e^{\eta\xi/d}\Bigr)
    & \hbox{as} \quad \xi \to -\infty\,, \\
    e^{\gamma(\xi-\xi_0)} + \cO(1)
    & \hbox{as} \quad \xi \to +\infty\,, \\
    \end{cases}
\]
where $\xi_0 \in \R$, $\gamma = r/c$, and $\eta = \min(\mu,\lambda)$.
Setting $\cU(\xi) = u(\Phi(\xi))$, $\cV(\xi) = v(\Phi(\xi))$, we 
obtain the desired solution of \eqref{Ode1} satisfying the boundary 
conditions \eqref{BC2}. Note that
\begin{align*}
  \cU(\xi) \,&=\, 1 - d + \alpha e^{\mu \xi/d} + \frac{d\mu}{\mu-\lambda}
  \Bigl(e^{\lambda \xi/d} - e^{\mu \xi/d}\Bigr) + \cO\Bigl(e^{2\eta\xi/d}
  \Bigr)\,, \\ 
  \cV(\xi) \,&=\, 1 - e^{\lambda \xi/d} + \cO\Bigl(e^{(\lambda+\eta)\xi/d}
  \Bigr)\,, \qquad \hbox{as }\xi \to -\infty\,,
\end{align*}
for some $\alpha \in \R$, whereas the asymptotic behavior 
\eqref{AsymUV2} as $\xi \to +\infty$ is unchanged. This concludes
the proof of Theorem~\ref{main2}. \QED

\section{Asymptotic analysis of slowly propagating fronts}\label{sec3}

Perhaps the most striking aspect of Theorems~\ref{main1} and \ref{main2} is the
absence of a minimal speed for the monotone traveling waves of
system~\eqref{redGGsys}.  To understand what happens in the singular limit
$c \to 0$, we compute in this section the leading term of a (formal) asymptotic
expansion of the front profile. We do not feel the necessity of rigorous proofs
at this stage, but we provide numerical illustrations supporting our
arguments. We always assume that $d > 1$, and we consider propagation fronts
connecting the infected state $(\cU_-,\cV_-) = (0,1)$ and the healthy state
$(\cU_+,\cV_+) = (1,0)$ of system~\eqref{Ode1}. As is explained in
Section~\ref{sec2}, such fronts correspond to solutions $(u,v,w)$ of the 
desingularized system \eqref{Ode3} satisfying the asymptotic conditions \eqref{BC3}.

If the parameter $c > 0$ is very small, the first equation in \eqref{Ode3} suggests
that the function $u$ is a {\em fast variable} in the sense of geometric singular 
perturbation theory \cite{Fen79}. Its transition from the initial value $0$ to the 
final value $1$ should occur in a small interval of size $\cO(c)$ centered at some
point $y \in \R$, which we assume to be the origin $y = 0$. So, in a first approximation, 
we expect that the function $v$ is close to a function $v_0$ satisfying
\begin{equation}\label{v0def}
\begin{array}{lll}
  v_0''(y) + c\,v_0'(y) + r v_0(y)\bigl(1-v_0(y)\bigr) \,=\, 0\,,
  & \hbox{ if} & y < 0\,, \\[1mm]  
  v_0''(y) + c\,v_0'(y) \,=\, 0\,, & \hbox{ if} & y > 0\,.  
\end{array}
\end{equation}

\begin{lem}\label{lem:v0}
If $0 < c < 2\sqrt{r}$, there exists a unique decreasing function $v_0 : \R \to 
(0,1)$ of class $C^{1,1}$ satisfying Eq.~\eqref{v0def} as well as the boundary 
conditions $v_0(-\infty) = 1$, $v_0(+\infty) = 0$. Moreover, one has
\begin{equation}\label{v00}
  v_0(0) \,=\, 1 - \frac{c}{c+\lambda} + \cO(c^2)\,, \qquad \hbox{as }~
  c \to 0\,,
\end{equation}
where $\lambda$ is given by \eqref{lamudef}. 
\end{lem}

\begin{proof}
Let $\phi$ be a decreasing solution of the Fisher--KPP equation \eqref{FKPP} such
that $\phi(y) \to 1$ as $y \to -\infty$. Since $c < 2\sqrt{r}$, we know that
$\phi$ does not stay positive, so there exists a unique $y_0 \in \R$ such that
$\phi(y_0) = 0$ and $\phi(y) > 0$ for all $y < y_0$.

We next consider the smooth function $\psi : (-\infty,y_0] \to \R$ defined by 
$\psi(y) = \phi'(y) + c\phi(y)$. We have $\psi'(y) = -r\phi(y)(1-\phi(y)) < 0$ 
for all $y < y_0$, whereas $\psi(-\infty) = c > 0$ and $\psi(y_0) = \phi'(y_0) < 0$. 
So there exists a unique $y_1 < y_0$ such that $\psi(y_1) = 0$, and after 
a suitable translation of the variable $y$ we can assume that $y_1 = 0$. 
If we now define
\begin{equation}\label{v0exp}
  v_0(y) \,=\, \begin{cases} \phi(y) & \hbox{if } y \le 0\,,\\
  \phi(0)e^{-cy} & \hbox{if } y \ge 0\,,
  \end{cases}
\end{equation}
we see that $v_0 \in C^{1,1}(\R)$, because $\phi'(0) + c\phi(0) = \psi(0) = 0$, 
and that $v_0$ satisfies \eqref{v0def} together with the desired boundary conditions.
Finally, we have as in \eqref{asym1} 
\begin{equation*}
  v(y) \,=\, 1 - e^{\lambda (y+y_2)} + \cO\Bigl(e^{2\lambda(y+y_2)}\Bigr)\,, 
  \qquad \hbox{as }~ y \to -\infty\,,
\end{equation*}
for some translation parameter $y_2 \in \R$. Neglecting the higher order
terms, we obtain the relation $0 = v'(0) + cv(0) = c - (c+\lambda)e^{\lambda y_2}$ 
which determines $y_2$, and we arrive at \eqref{v00}. 
\end{proof}

\begin{rem}\label{rem:v0}
There is no explicit formula for the function $v_0$ in Lemma~\ref{lem:v0}, 
but in the asymptotic regime where $c \ll 1$ one has $v_0 = \hat v_0 + 
\cO(c^2)$ where 
\begin{equation}\label{v0hat}
  \hat v_0(y) \,=\, \begin{cases} 1 - {\DS\frac{c}{c+\lambda}}
  \,e^{\lambda y} & \hbox{if } y \le 0\,,\\[2mm]
  {\DS\frac{\lambda}{c+\lambda}}\,e^{-cy} & \hbox{if } y \ge 0\,.
  \end{cases}
\end{equation}
Note that $\hat v_0 \in C^{1,1}(\R)$ and $\hat v_0(-\infty) = 1$, 
$\hat v_0(+\infty) = 0$.
\end{rem}

We next construct the leading order approximation of the function $u$.  
We assume that $c > 0$ is small enough so that $v_0(0) > 1/d$, which 
is possible in view of \eqref{v00}, and we denote $b = d v_0(0) - 1 > 0$. 
Observing that $dv_0(y) \approx dv_0(0) = 1 + b$ if $|y| = \cO(c)$, we 
postulate that $u$ is well approximated by a function $u_0$ satisfying 
the simplified equation
\begin{equation}\label{u0eq}
   u_0'(y) \,=\,  \frac{1}{c}\,u_0(y)(1-u_0(y))\bigl(b + u_0(y)\bigr)\,, 
   \qquad y \in \R\,,  
\end{equation}
together with the boundary conditions $u_0(-\infty) = 0$, $u_0(+\infty) = 1$. 
The solution of \eqref{u0eq} is implicitly given by the relation
\begin{equation}\label{u0exp}
  \frac{u_0(y)^{1+b}}{\bigl(1-u_0(y)\bigr)^b \bigl(b + u_0(y)\bigr)}
  \,=\, \frac{\alpha^{1+b}}{b}\,\exp\biggl(\frac{b(1+b)}{c}\,y\biggr)\,, 
  \qquad y \in \R\,,
\end{equation}
where $\alpha > 0$ is an integration constant which amounts to fixing the 
value $u_0(0) \in (0,1)$. This constant can be determined, for instance, by 
imposing the relation 
\begin{equation}\label{capp}
  c \,=\, r \int_{\R} \hat v_0(y)\bigl(1 - \hat v_0(y)\bigr)
  \bigl(1 - u_0(y)\bigr)\dd y\,, 
\end{equation}
which is the analogue of \eqref{cformula} at our level of approximation.
For later use, we also note that $u_0(y) = \alpha\,e^{by/c} + \cO(e^{2by/c})$
in the asymptotic regime where $y \to -\infty$. 

The approximate solution $(u_0,v_0)$ of \eqref{Ode2} constructed so
far describes relatively well the asymptotic region $y \to -\infty$
and the central region where the transition occurs from a neighborhood
of $(0,1)$ to the vicinity of $(1,0)$. However, this first-order
approximation is not realistic when $y > 0$ is large, because
$1-u_0(y)$ and $v_0(y)$ decay exponentially to zero as $y \to +\infty$, 
in sharp contrast with \eqref{asym2}. Nevertheless, it is highly interesting 
at this point to return to the original variables and to compute the 
corresponding approximate solution $(\cU_0,\cV_0)$ of \eqref{Ode1}. 
First of all, it is important to realize that $(\cU_0,\cV_0)$ is 
a {\em sharp front}, associated with some finite value $\bar\xi < +\infty$. 
Indeed, in view of \eqref{u0eq} the change of variables \eqref{xidef} becomes
\begin{equation}\label{chang1}
  \frac{\D \xi}{\D y} \,=\, 1 - u_0(y) \,=\, \frac{c u_0'(y)}{
  u_0(y)(b+u_0(y))}\,, \qquad y \in \R\,,
\end{equation}
so that
\begin{equation}\label{xiapp}
  \xi \,=\, \Phi^{-1}(y) \,=\, \frac{c}{b}\,\log\biggl(\frac{(1{+}b)
  u_0(y)}{b+u_0(y)}\biggr)\,, \qquad y \in \R\,.
\end{equation}
Here we have normalized things so that $\bar\xi = 0$, which means that the map
$\Phi : (-\infty,0) \to \R$ is a diffeomorphism. Moreover $\Phi^{-1}(y) 
\approx y + \kappa$ as $y \to -\infty$, where $\kappa = \frac{c}{b}\log\bigl(
\frac{1+b}{b}\,\alpha\bigr)$. It follows immediately from \eqref{xiapp} that
\begin{equation}\label{U0exp}
  \cU_0(\xi) \,:=\, u_0(\Phi(\xi)) \,=\, \frac{b\,e^{b\xi/c}}{1+b - 
  e^{b\xi/c}}\,, \qquad \xi \in (-\infty,0)\,.
\end{equation}
Remarkably, this expression does not involve the constant $\alpha$ in \eqref{u0exp}. 

Using \eqref{U0exp}, we can in turn compute the map $\Phi$ more explicitly. To 
this end, we write \eqref{chang1} in the equivalent form
\begin{equation}\label{chang2}
  \frac{\D y}{\D \xi} \,=\, \frac{1}{1 - \cU_0(\xi)} \,=\, \frac{1}{1+b}
  \,\frac{1+b - e^{b\xi/c}}{1 - e^{b\xi/c}}\,, \qquad \xi \in (-\infty,0)\,,
\end{equation}
and we easily deduce
\begin{equation}\label{Phiapp}
  y \,=\, \Phi(\xi) \,=\, \xi - \frac{c}{1+b}\,\log\Bigl(1 -  e^{b\xi/c}\Bigr)
  - \kappa\,, \qquad \xi \in (-\infty,0)\,.
\end{equation}
It follows that
\begin{equation}\label{V0exp}
   \cV_0(\xi) \,:=\, \hat v_0(\Phi(\xi)) \,=\, \begin{cases} 
   1 - {\DS\frac{c}{c+\lambda}}\,e^{\lambda\Phi(\xi)} & \hbox{if } \xi \le \xi_*\,,\\[2mm]
  {\DS\frac{\lambda}{c+\lambda}}\,e^{-c\Phi(\xi)} & \hbox{if } \xi_* \le \xi < 0\,, 
  \end{cases}
\end{equation}
where $\xi_* = \Phi^{-1}(0) \in (-\infty,0)$. Note that $\xi_*$ depends on 
$\kappa$, hence on the constant $\alpha$ in \eqref{u0exp}. Figure~\ref{fig4} 
shows that the approximations \eqref{U0exp}, \eqref{V0exp} are remarkably 
accurate, even at moderately small speeds such as $c = 0.2$. 

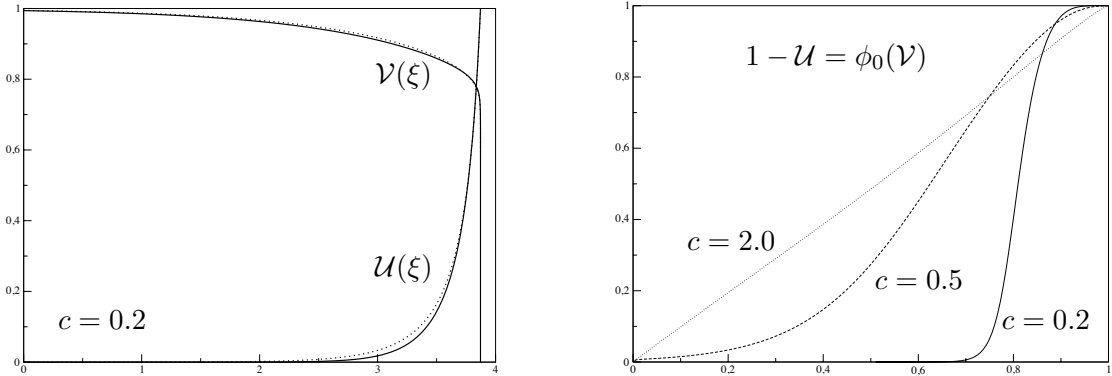
\begin{figure}[ht]
  \begin{center}
  \begin{picture}(480,160)
  \put(20,0){\includegraphics[width=0.43\textwidth]{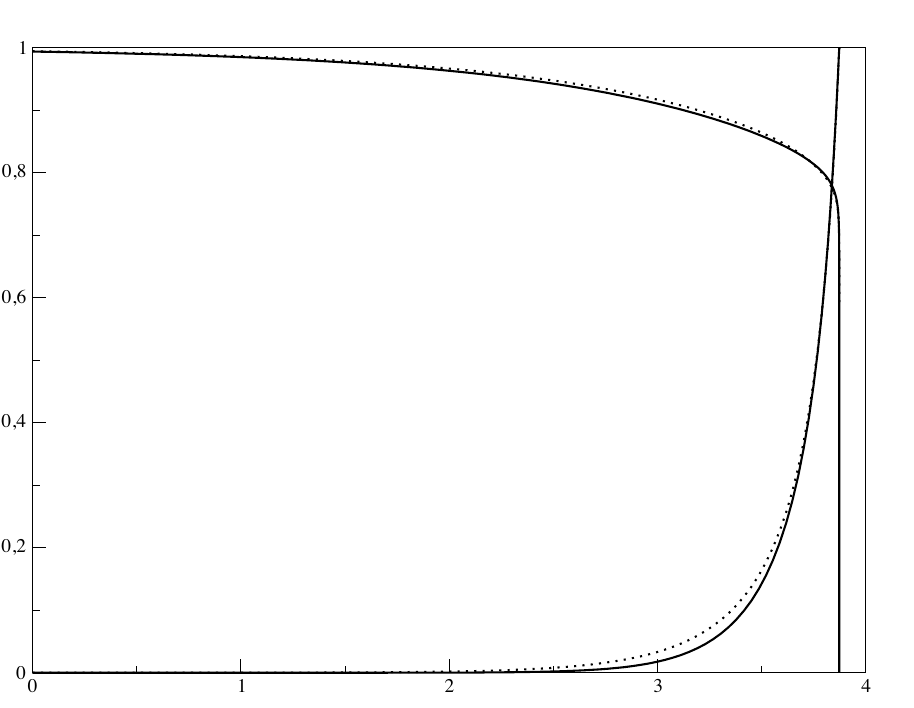}}
  \put(250,0){\includegraphics[width=0.43\textwidth]{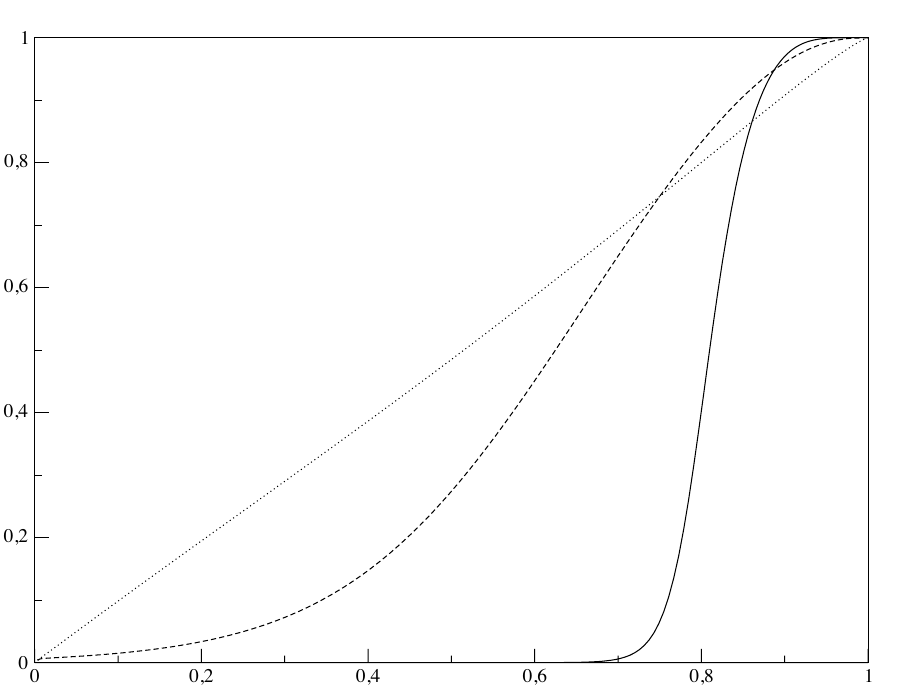}}
  \put(160,40){$\cU(\xi)$}
  \put(160,113){$\cV(\xi)$}
  \put(40,20){$c=0.2$}
  \put(278,50){$c=2.0$}
  \put(349,35){$c=0.5$}
  \put(397,20){$c=0.2$}
  \put(300,120){$1-\cU = \phi_0(\cV)$}
  \end{picture}
  \caption{{\small The sharp profile of the propagation front $(\cU,\cV)$ for $d = 2$, 
  $r = 1$, $c = 0.2$ is represented in the left picture (solid lines), 
  as well as the approximations given by \eqref{U0exp}, \eqref{V0exp} 
  (dotted lines). For the same values of $d,r$, the right picture shows 
  that the effective diffusion coefficient \eqref{Deff2} depends strongly 
  on the speed parameter.} \label{fig4}}
  \end{center}
\end{figure}

Let $\phi_0 : (0,1) \to (0,1)$ be the approximate diffusion coefficient defined
by 
\begin{equation}\label{Deff2}
  1 - \cU_0(\xi) \,=\, \phi_0\bigl(\cV_0(\xi)\bigr)\,, \qquad \xi \in (-\infty,0)\,.
\end{equation}
Equivalently, we have $1 - u_0(y) = \phi_0(\hat v_0(y))$ for all $y \in \R$.  
The function $\phi_0$ can be evaluated using the formulas \eqref{U0exp}, \eqref{V0exp}, 
and is expected to give a good approximation of the effective diffusion coefficient 
\eqref{Deff} when $c \ll 1$. It is straightforward to verify that
\begin{equation}\label{Dasym}
  \phi_0(v) \,\approx\, \beta \biggl(\frac{c+\lambda}{\lambda}\,v\biggr)^{\frac{1+b}{c^2}}
  \quad \hbox{as } v \to 0\,, \qquad \hbox{where}\quad \beta^b \,=\,
  \frac{b}{b+1}\,\frac{1}{\alpha^{1+b}}\,.
\end{equation}
Since $1+b = d + \cO(c)$, this means that the exponent $(1+b)/c^2$ in 
\eqref{Dasym} is very large when $c \ll 1$, so that the function $\phi_0$ is 
extremely flat near the origin, see Figure~\ref{fig4}b. This in turn 
explains why Theorem~\ref{main1} does not conflict with classical results 
establishing the existence of a minimal speed for the traveling waves of 
scalar equations with degenerate diffusion. To see this, consider the 
model equation
\begin{equation}\label{Dm}
  \partial_t V \,=\, D\,\partial_x \bigl(V^m \partial_x V\bigr) + 
  r V (1 - V)\,, 
\end{equation}
where $D,r$ are positive constants and $m \in \N\setminus\{0\}$. It is known
\cite{AtkiReutRidl81} that the minimal speed $c_*$ of the traveling waves
for \eqref{Dm} satisfies
\[
  \frac{2Dr}{(m+1)(m+2)} \,\le\, c_* \,\le\, \frac{2Dr}{m(m+1)}\,, 
\]
so that $c_* \sim \sqrt{2Dr}/m$ as $m \to +\infty$. Although \eqref{Dasym} is
only an asymptotic formula valid for $v \to 0$, this suggests that the minimal
speed for the scalar equation \eqref{GGscalar} where $\phi = \phi_0$ can be
compared to the minimal speed for \eqref{Dm} where $m = (1+b)/c^2$.  The
latter is proportional to $1/m \approx c^2/d$, and should therefore become
smaller than $c$ when $c \ll 1$. Summarizing, given any $c > 0$ (no matter how
small), the effective diffusion coefficient $\phi$ defined by \eqref{Deff} is so
flat near the origin that the minimal speed $c_*$ associated with the scalar
equation \eqref{GGscalar} always satisfies $c_* \le c$.

\section{Conclusions and perspectives}\label{sec4}

We conclude this paper with a list of possible questions
that are, in our opinion, worth investigating in the future.

\subsection{External parameters $d$ and $r$}\label{subsec41}

Tumor growth is dependent on the complex interactive dynamics of many different
factors, including competitive effects (here, described by the parameter $d$)
and growth factors (here, by the parameter $r$). Invasion fronts for the reduced
Gatenby--Gawlinski model \eqref{redGGsys} have been explored in detail,
providing a complete existence result for all positive values of $d$ and
$r$. The critical threshold $d=1$ separates two different scenarios
(heterogeneous versus homogeneous invasion) with a bifurcation appearing already
at the level of the equilibria. In contrast, varying $r$ does not lead to
qualitatively different behaviors. These parameters are both relevant from an
``oncological'' point of view, as they describe two distinct properties of the
system: increasing $d$ enhances the competitivity of the cancerous cells against
the healthy tissue, while $r$ describes the reproduction activity of the tumor
cells alone. In principle, one should be able to fix appropriate values for these
parameters by comparing the shape of the propagation fronts of \eqref{redGGsys}
with experimental data or predictions from more complete models.

\subsection{Stability and minimal speed}\label{subsec42}

While our results prove the existence of propagation fronts for any positive
value of the speed parameter $c$, this does not mean that there is no 
minimal speed for system~\eqref{redGGsys}. First, we do not have any information
so far on the stability of the fronts constructed in Theorems~\ref{main1} 
and \ref{main2}. In fact, there are even mathematical issues concerning the 
Cauchy problem itself, see Section~\ref{subsec43} below. So, it may well be
that the propagation fronts of system~\eqref{redGGsys} are unstable in some
parameter regimes, for instance when $c > 0$ is sufficiently small. Also, 
we do not know which propagation front is selected from Heaviside-type
initial data, and it is therefore conceivable that a minimal speed arises 
in that context too. This last equation is of course much more difficult
than in the scalar case, as system~\eqref{redGGsys} has no maximum principle. 

\subsection{Well-posedness of the Cauchy problem}\label{subsec43}

Due to the presence of degeneracy in the second equation of \eqref{redGGsys},
establishing the well-posedness of the Cauchy problem on the real line $\R$
is a delicate issue. The crucial question is of course whether the threshold
value $\bar U$ such that $f(\bar U)=0$ (in our case, $\bar U=1$) is reached
somewhere.  In that case, heuristic arguments suggest that the problem is not
well-posed in the classical sense, so that some appropriate weak formulation has
to be used. It happens that our reduction from \eqref{GGsys} to
\eqref{redGGsys}, which eliminates the intermediate agent $W$, drastically
increases the stiffness of the diffusion degeneracy. In particular, if $V=0$ and
$U\equiv 1$ in some region, i.e. tumor cells are absent and healthy tissue is at
carrying capacity, propagation is expected to be completely blocked, thus
preventing any invasion mechanism. To our knowledge, well-posedness in some appropriate 
weak framework has not been explored yet, but partial results can be found in 
\cite{BarrNurn02, BarrDeck12}. Of course if the critical value $\bar U = 1$ 
is never reached, existence of a unique classical solution is expected, and 
can be proved by standard techniques. 

On the other hand, an extended version of the Gatenby-Gawlinski model (see
\cite{McGiEtAl14}) has been analyzed in \cite{TaoTell16} for the case of a
multi-dimensional bounded domain with smooth boundary and zero-flux boundary
conditions. In view of the correspondence \eqref{redGGgenGG}, the assumptions in
\cite{TaoTell16} reduce, in the case of \eqref{redGGsys}, to the single hypothesis
$d<1$, which corresponds to heterogeneous invasion. Under such conditions, it
can be proved that, for any positive time $t$, the component $U$ is bounded away
from the critical value $1$, so that the problem possesses a unique classical
solution for initial data $U_0\in(0,1)$ and $V_0>0$. Incidentally, let us remark
that the assumption that $V_0 > 0$ corresponds to the initial presence of tumor
spread everywhere in the healthy tissue, which is clearly questionable from a
biological perspective.  A different approach has been proposed in
\cite{MarkMeraSuru13}, where the authors show local and global existence
invoking an iterative strategy. This approach imposes no restriction on the 
values of the coupling parameter $d$, but it is crucial to assume that 
$U_0\leq \theta$ for some $\theta < 1$. Summarizing, the fundamental question of 
well-posedness for system \eqref{redGGsys} remains currently unsolved in its 
full generality. 

Ideally, the target is to come back to the complete Gatenby--Gawlinski model
\eqref{GGsys}.  Actually, the results quoted above
\cite{MarkMeraSuru13,TaoTell16} do apply to system \eqref{GGsys}. The mediation of
the acid variable --satisfying a linear parabolic equation with dissipation and
external forcing-- increases the possibility of recovering some sort of
classical framework, even if we cannot quote any result of this nature. 
Some weak formulation may also be needed to prove well-posedness in the 
sense of Hadamard --i.e. existence, uniqueness and continuous dependence--, 
but we are not aware of any complete result in that direction either. 

\subsection{Propagating fronts for the complete model}\label{subsec44}

A comprehensive study of existence of traveling waves for the original
Gatenby--Gawlinski model \eqref{GGsys} is currently not available in the
literature. Some partial results, based on singular perturbation theory,
have been presented in \cite{DaviEtAl18}. In any case, there is a clear
computational evidence of existence of such fronts, see
\cite{McGiEtAl14,MoscSime19}. In addition, numerical simulations for
\eqref{GGsys} indicate that the coupling with the acid equation is crucial 
for the existence of a strictly positive minimal speed, which should 
correspond to a sharp front in the regime $d>1$. As a final remark, let us 
observe that many properties used in the present work --for example, the 
monotonicity in Lemma~\ref{lem:monotone}-- are specific to the reduced model 
\eqref{redGGsys} and cannot be easily generalized to the original 
Gatenby--Gawlinski system \eqref{GGsys}.


\end{document}